\documentclass[reqno]{amsart}
\usepackage{amssymb}
\usepackage{graphicx}

\usepackage[usenames, dvipsnames]{color}
\usepackage{verbatim}
\usepackage{mathrsfs}
\usepackage{bm}
\usepackage{cite}
\usepackage{bbm}

\numberwithin{equation}{section}

\newtheorem{theorem}{Theorem}[section]
\newtheorem{corollary}[theorem]{Corollary}
\newtheorem{lemma}[theorem]{Lemma}
\newtheorem{prop}[theorem]{Proposition}

\theoremstyle{definition}
\newtheorem{remark}[theorem]{Remark}

\theoremstyle{definition}

\theoremstyle{definition}

\makeatletter
\def\dashint{\operatorname%
{\,\,\text{\bf-}\kern-.98em\DOTSI\intop\ilimits@\!\!}}
\makeatother

\def\\det{\text{\det}}

\def\Xint#1{\mathchoice
 {\XXint\displaystyle\textstyle{#1}}%
 {\XXint\textstyle\scriptstyle{#1}}%
 {\XXint\scriptstyle\scriptscriptstyle{#1}}%
 {\XXint\scriptscriptstyle\scriptscriptstyle{#1}}%
 \!\int}
\def\XXint#1#2#3{{\setbox0=\hbox{$#1{#2#3}{\int}$}
  \vcenter{\hbox{$#2#3$}}\kern-.5\wd0}}

\def\dashint{\Xint-}

\def\.5{\frac{1}{2}}

\newcommand{\RN}[1]{%
  \textup{\uppercase\expandafter{\romannumeral#1}}%
}

\renewcommand{\epsilon}{\varepsilon}

\newcounter{marnote}

\begin{document}

\title[Local regularity for weighted singular parabolic equations]{Local regularity for solutions to quasi-linear singular parabolic equations with anisotropic weights}

\author[C.X. Miao]{Changxing Miao}
\address[C.X. Miao] {1. Beijing Computational Science Research Center, Beijing 100193, China.}
\address{2. Institute of Applied Physics and Computational Mathematics, P.O. Box 8009, Beijing, 100088, China.}
\email{miao\_changxing@iapcm.ac.cn}

\author[Z.W. Zhao]{Zhiwen Zhao}

\address[Z.W. Zhao]{Beijing Computational Science Research Center, Beijing 100193, China.}
\email{zwzhao365@163.com}

%\footnote{}

\date{\today} % delete this line to display the current date

%%% BEGIN DOCUMENT

\maketitle
%\tableofcontents
\begin{abstract}
This paper develops a concise procedure for the study on local behavior of solutions to anisotropically weighted quasi-linear singular parabolic equations of $p$-Laplacian type, which is realized by improving the energy inequalities and applying intrinsic scaling factor to the De Giorgi truncation method. In particular, it also presents a new proof for local H\"{o}lder continuity of the solution in the unweighted case.

%, which promotes the application of the De Giorgi theory.

\end{abstract}

\maketitle
%\date{}
%\maketitle
%{\bf Abstract}

%{\bf{Keywords:}} local H\"{o}lder regularity; quasilinear singular parabolic equations; anisotropic weights

%\noindent{\bf{MSC numbers}}: {35K92; 35B40; 35B65.}

\section{Introduction}%\label{intro}

Suppose that $\Omega$ is a smooth bounded domain in $\mathbb{R}^{n}$ with $n\geq2$. For $T>0$, write $\Omega_{T}:=\Omega\times(-T,0]$ and let $\partial_{pa}\Omega_{T}$ represent the standard parabolic boundary of $\Omega_{T}$. In this paper, we consider the following anisotropically weighted quasilinear parabolic equations with the Dirichlet condition:
\begin{align}\label{PO001}
\begin{cases}
w_{1}\partial_{t}u-\mathrm{div}(w_{2}\mathbf{a}(x,t,u,\nabla u))=w_{2}b(x,t,u,\nabla u),& \mathrm{in}\;\Omega_{T},\\
u=\psi,&\mathrm{on}\;\partial_{pa}\Omega_{T},
\end{cases}
\end{align}
where $\psi\in L^{\infty}(\partial_{pa}\Omega_{T})$, $w_{1}=|x'|^{\theta_{1}}|x|^{\theta_{2}}$, $w_{2}=|x'|^{\theta_{3}}|x|^{\theta_{4}}$, $x'=(x_{1},...,x_{n-1})$, the ranges of $\theta_{i}$, $i=1,2,3,4$ are given in the following theorems, the functions $\mathbf{a}:\Omega_{T}\times\mathbb{R}^{n+1}\rightarrow\mathbb{R}^{n}$ and $b:\Omega_{T}\times\mathbb{R}^{n+1}\rightarrow\mathbb{R}$ are assumed to be only measurable and satisfy the structure conditions as follows: for $p>1$ and $(x.t)\in\Omega_{T}$,
\begin{itemize}
{\it
\item[$(\mathbf{H1})$] $\mathbf{a}(x,t,u,\nabla u)\cdot\nabla u\geq\lambda_{1}|\nabla u|^{p}-\phi_{1}(x,t)$,
\item[$(\mathbf{H2})$] $|\mathbf{a}(x,t,u,\nabla u)|\leq\lambda_{2}|\nabla u|^{p-1}+\phi_{2}(x,t)$,
\item[$(\mathbf{H3})$] $|b(x,t,u,\nabla u)|\leq\lambda_{3}|\nabla u|^{p-1}+\phi_{3}(x,t)$.}
\end{itemize}
Here $\lambda_{i}>0$, $i=1,2,$ $\lambda_{3}\geq0$, and $\phi_{i}$, $i=1,2,3$ are nonnegative functions subject to the condition of
\begin{align}\label{ZE90}
\|\phi\|_{L^{l_{0}}(\Omega_{T},w_{2})}:=\bigg(\int_{\Omega_{T}}|\phi|^{l_{0}}w_{2}dxdt\bigg)^{\frac{1}{l_{0}}}<\infty,\quad \phi:=\phi_{1}+\phi_{2}^{\frac{p}{p-1}}+\phi_{3}^{\frac{p}{p-1}},
\end{align}
where $l_{0}$ satisfies
\begin{align}\label{E01}
l_{0}>\frac{n+p+\theta_{1}+\theta_{2}}{p}.
\end{align}
Remark the the anisotropy of the weight $|x'|^{\theta_{1}}|x|^{\theta_{2}}$ is induced by $|x'|^{\theta_{1}}$. From the perspective of geometry, $|x'|^{\theta_{1}}$ denotes a singular or degenerate line compared to the isotropic weight $|x|^{\theta_{2}}$, which is also called the Caffarelli-Kohn-Nirenberg weight due to the well-known work \cite{CKN1984}.

The primary objective of this paper is to study the local behavior for solutions to problem \eqref{PO001} with $1<p<2$. For that purpose, we succeed in improving the energy inequalities and applying intrinsic scaling factor to the De Giorgi truncation method of parabolic version, which makes us achieve a more clear and terse proof for local H\"{o}lder continuity of the solution than the classical proof procedures given in Chapter IV of \cite{D1993}. Specifically, the proof in Chapter IV of \cite{D1993} involved quite complex procedures based on comprehensive applications of the variable substitution, local energy and logarithmic estimates, especially requiring a great deal of computations in handling some integral inequalities. By contrast, our proof is very concise and mainly consists of three main ingredients in the De Giorgi truncation method including the oscillation improvement, the decay estimates and the expansion of the distribution function of solution in time, see Lemmas \ref{LEM0035}, \ref{lem005} and \ref{LEM006} below for more details. In \cite{MZ202301} there is a clear comparison that the main difference between elliptic and parabolic versions of the De Giorgi truncation method lies in the expansion in time, which was previously presented in Lemmas 3.7 and 3.8 of \cite{JX2022} for $p=2$. The proof of $p=2$ in \cite{JX2022,MZ202301} substantially utilized the homogeneity between the space and time variables and thus can be regarded as a natural extension of the classical De Giorgi method of elliptic version created in \cite{D1957}.

However, when $p\neq2$, the situation becomes markedly different due to the inhomogeneity of time-space variables such that the proof of $p=2$ cannot be applied to the case of $p\neq2$. In order to overcome this difficulty, in the case of $1<p<2$ we first improve the energy inequalities and then choose suitable scaling factor to stretch the space variable such that the homogeneity between the space and time variables is recovered, which allows us to perform the De Giorgi proof procedure of parabolic version in the rescaled cylinders. We stress that this idea is inspired by intrinsic scaling technique developed in \cite{D1993,CD1988,D1986}. In a word, the major novelty of this paper lies in providing a new and elegant proof for local H\"{o}lder continuity of solutions to quasilinear singular parabolic equations and the significance is that our investigation not only promotes the application of intrinsic scaling technique, but also further develops the De Giorgi theory such that it can be used to deal with quasilinear parabolic equations of $p$-Laplacian type.

Before stating the definition of weak solution to problem \eqref{PO001}, we first introduce the required weighted spaces. For a given weight $w$, let $L^{p}(\Omega,w)$, $L^{p}(\Omega_{T},w)$ and $W^{1,p}(\Omega,w)$ be, respectively, the weighted $L^{p}$ spaces and Sobolev spaces endowed with the following norms:
\begin{align*}
\begin{cases}
\|u\|_{L^{p}(\Omega,w)}=\left(\int_{\Omega}|u|^{p}wdx\right)^{\frac{1}{p}},\quad\|u\|_{L^{p}(\Omega_{T},w)}=\big(\int_{\Omega_{T}}|u|^{p}wdxdt\big)^{\frac{1}{p}},\vspace{0.3ex}\notag\\
\|u\|_{W^{1,p}(\Omega,w)}=\left(\int_{\Omega}|u|^{p}wdx\right)^{\frac{1}{p}}+\left(\int_{\Omega}|\nabla u|^{p}wdx\right)^{\frac{1}{p}}.
\end{cases}
\end{align*}
We say that $u\in C((-T,0];L^{2}(\Omega,w_{1}))\cap L^{p}((-T,0);W^{1,p}(\Omega,w_{2}))$ is a weak solution of problem \eqref{PO001}, provided that for any $-T<t_{1}<t_{2}\leq0$,
\begin{align*}
&\int_{\Omega}u\varphi w_{1}dx\Big|_{t_{1}}^{t_{2}}+\int_{t_{1}}^{t_{2}}\int_{\Omega}(-u\partial_{t}\varphi w_{1}+w_{2}\mathbf{a}(x,t,u,\nabla u)\cdot\nabla\varphi)dxdt\notag\\
&=\int_{t_{1}}^{t_{2}}\int_{\Omega}b(x,t,u,\nabla u)\varphi w_{2}dxdt,
\end{align*}
for every $\varphi\in W^{1,2}((0,T);L^{2}(\Omega,w_{1}))\cap L^{p}((0,T);W^{1,p}_{0}(\Omega,w_{2}))$.

Denote
\begin{align*}
\begin{cases}
\mathcal{A}=\{(\theta_{1},\theta_{2}): \theta_{1}>-(n-1),\,\theta_{2}\geq0\},\\
\mathcal{B}=\{(\theta_{1},\theta_{2}):\theta_{1}>-(n-1),\,\theta_{2}<0,\,\theta_{1}+\theta_{2}>-n\},\\
\mathcal{C}_{p}=\{(\theta_{1},\theta_{2}):\theta_{1}<(n-1)(p-1),\,\theta_{2}\leq0\},\\
\mathcal{D}_{p}=\{(\theta_{1},\theta_{2}):\theta_{1}<(n-1)(p-1),\,\theta_{2}>0,\,\theta_{1}+\theta_{2}<n(p-1)\}.
\end{cases}
\end{align*}
Define the exponent conditions as follows:
\begin{itemize}
{\it
\item[$(\mathbf{K1})$] $(\theta_{1},\theta_{2})\in[(\mathcal{A}\cup\mathcal{B})\cap(C_{p}\cup\mathcal{D}_{p})]\cup\{\theta_{1}=0,\,\theta_{2}\geq n(p-1)\}$,\;$(\theta_{3},\theta_{4})\in\mathcal{A}\cup\mathcal{B}$,
\item[$(\mathbf{K2})$] $\theta_{1}+\theta_{2}>p-n,\;\theta_{1}\geq\theta_{3},\;\theta_{1}+\theta_{2}\geq\theta_{3}+\theta_{4},\;\theta_{3}+\min\{0,\theta_{4}\}>1-n$.}
\end{itemize}
With regard to the implications of $\mathrm{(}\mathbf{K1}\mathrm{)}$--$\mathrm{(}\mathbf{K2}\mathrm{)}$, see the introduction in \cite{MZ202303} for detailed interpretations. We additionally assume that
\begin{align}\label{CON01}
\lambda_{1}\gg \lambda_{3}.
\end{align}
This condition is assumed to ensure the establishment of global $L^{\infty}$ estimate for solution to problem \eqref{PO001}, see Remark \ref{REM09} below for more detailed explanations.

In the following, we utilize $C$ to represent a universal constant depending only upon the data including $n,p,l_{0},\|\psi\|_{L^{\infty}(\partial_{pa}\Omega_{T})},\|\phi\|_{L^{l_{0}}(\Omega_{T},w_{2})}$, $\lambda_{i}$, $i=1,2,3,$ and $\theta_{i}$, $i=1,2,3,4,$ whose value may vary from line to line.

Write
\begin{align}\label{VAR01}
\varepsilon_{0}:=\frac{(p+\vartheta)(p(l_{0}-1)-n-\theta_{1}-\theta_{2})}{pl_{0}(2+\vartheta)+(p-2)(n+\theta_{3}+\theta_{4})},\quad\vartheta:=\theta_{1}+\theta_{2}-\theta_{3}-\theta_{4}.
\end{align}
To begin with, we give a precise characterization for the asymptotic behavior of solution to problem \eqref{PO001} near the singular or degenerate point of the weights by capturing an explicit upper bound on the decay rate exponent.
\begin{theorem}\label{ZWTHM90}
Assume that $1<p<2$, $n\geq2$, $\mathrm{(}\mathbf{H1}\mathrm{)}$--$\mathrm{(}\mathbf{H3}\mathrm{)}$, \eqref{ZE90}--\eqref{CON01} and $\mathrm{(}\mathbf{K1}\mathrm{)}$--$\mathrm{(}\mathbf{K2}\mathrm{)}$ hold. Let $u$ be a weak solution of problem \eqref{PO001} with $\Omega\times(-T,0]=B_{1}\times(-1,0]$. Then there exists a small constant $0<\alpha\leq\varepsilon_{0}$ depending only on the above data, such that for any fixed $t_{0}\in (-1/2,0]$ and all $(x,t)\in B_{1/2}\times(-1/2,t_{0}]$,
\begin{align}\label{QNAW001ZW}
u(x,t)=u(0,t_{0})+O(1)\big(|x|+|t-t_{0}|^{\frac{1}{p+\vartheta}}\big)^{\alpha},
\end{align}
where $\varepsilon_{0},\vartheta$ are defined by \eqref{VAR01}, $|O(1)|\leq C$ for some positive constant $C$ depending only on the above data.
\end{theorem}

\begin{remark}
From \eqref{E01} and $\mathrm{(}\mathbf{K2}\mathrm{)}$, we have
\begin{align*}
0\leq\vartheta<p(l_{0}-1)-n-\theta_{3}-\theta_{4},\quad p<pl_{0}-n-\theta_{1}-\theta_{2}<pl_{0}-2,
\end{align*}
which, together with the fact of $1<p<2$, leads to that
\begin{align*}
p(l_{0}-1)+2<&pl_{0}+\frac{(2-p)(pl_{0}-n-\theta_{3}-\theta_{4})}{p+\vartheta}\notag\\
\leq&pl_{0}+\frac{(2-p)(pl_{0}-n-\theta_{1}-\theta_{2})}{p}<2l_{0}-\frac{4}{p}+2<2l_{0}.
\end{align*}
This implies that
\begin{align*}
\frac{1}{2l_{0}}<\frac{p+\vartheta}{pl_{0}(2+\vartheta)+(p-2)(n+\theta_{3}+\theta_{4})}<\frac{1}{p(l_{0}-1)+2}.
\end{align*}
Therefore, by fixing the values of $n,p,l_{0}$, we deduce that $\varepsilon_{0}\rightarrow0$, as $\theta_{1}+\theta_{2}\rightarrow p(l_{0}-1)-n$. Furthermore, we have from Lemma \ref{PRO90} below that the exponent $\alpha$ in \eqref{QNAW001ZW} can achieve the upper bound $\varepsilon_{0}$ when $\varepsilon_{0}$ is chosen to be sufficiently small.
\end{remark}

\begin{remark}
Combining the results of $p>2$ in \cite{MZ202303}, we see that when the value of $\theta_{1}+\theta_{2}$ sufficiently approaches $p(l_{0}-1)-n$,
\begin{align*}
\alpha=\varepsilon_{0}=
\begin{cases}
\frac{(p+\vartheta)(p(l_{0}-1)-n-\theta_{1}-\theta_{2})}{pl_{0}(2+\vartheta)+(p-2)(n+\theta_{3}+\theta_{4})},&\mathrm{for}\;1<p<2,\vspace{0.5ex}\\
\frac{2(l_{0}-1)-n-\theta_{1}-\theta_{2}}{2l_{0}},&\mathrm{for}\;p=2,\vspace{0.5ex}\\
\frac{p(l_{0}-1)-n-\theta_{1}-\theta_{2}}{p(l_{0}-1)+2},&\mathrm{for}\;p>2.
\end{cases}
\end{align*}
This reveals a discrepancy phenomenon in terms of the form of $\varepsilon_{0}$ under different $p$ such that the case of $p=2$ becomes the critical case. Moreover, the complex form of $\varepsilon_{0}$ in the case of $1<p<2$ will result in more sophisticated effects of the weights on the smoothness of solution. On one hand, we first fix the values of $n,p,l_{0},\theta_{3},\theta_{4}$. Since $0<\varepsilon_{0}<1$ and $\varepsilon_{0}\searrow0$, as $\theta_{1}+\theta_{2}\nearrow p(l_{0}-1)-n$, then we obtain that $\varepsilon_{0}$ and $\frac{\varepsilon_{0}}{p+\vartheta}$ are all decreasing in $\theta_{1}+\theta_{2}$ when $\theta_{1}+\theta_{2}$ is sufficiently close to $p(l_{0}-1)-n$, which shows the weakening effect of the weights on the time-space regularity of $u$. On the other hand, when the values of $n,p,l_{0},\theta_{1},\theta_{2}$ are fixed, it follows from a direct computation that $\varepsilon_{0}$ is decreasing with respect to $\theta_{3}+\theta_{4}$, while $\frac{\varepsilon_{0}}{p+\vartheta}$ is increasing in $\theta_{3}+\theta_{4}$. In this case the weights will weaken the regularity of the solution in space variable and meanwhile enhance the regularity in time variable, as the value of $\theta_{3}+\theta_{4}$ increases.

\end{remark}

Now we consider the following non-homogeneous weighted parabolic $p$-Laplace equation:
\begin{align}\label{PO009}
\begin{cases}
w_{1}\partial_{t}u-\mathrm{div}(w_{2}|\nabla u|^{p-2}\nabla u)=w_{2}(\lambda_{3}|\nabla u|^{p-1}+\phi_{3}(x,t)),& \mathrm{in}\;\Omega_{T},\\
u=\psi,&\mathrm{on}\;\partial_{pa}\Omega_{T}.
\end{cases}
\end{align}
In the presence of the same type of single power-type weights, we further derive the H\"{o}lder estimates for solution to problem \eqref{PO009} as follows.
\begin{theorem}\label{THM060}
Let $1<p<2$, $n\geq2$, \eqref{ZE90}--\eqref{CON01} and $\mathrm{(}\mathbf{K1}\mathrm{)}$--$\mathrm{(}\mathbf{K2}\mathrm{)}$ hold. Assume that $u$ is a weak solution of problem \eqref{PO009} with $\Omega\times(-T,0]=B_{1}\times(-1,0]$. Then there exists a constant $0<\tilde{\alpha}<\frac{\alpha}{1+\alpha}$ with $\alpha$ determined by Theorem \ref{ZWTHM90}, such that if  $(w_{1},w_{2})=(|x'|^{\theta_{1}},|x'|^{\theta_{3}})$ or $(w_{1},w_{2})=(|x|^{\theta_{2}},|x|^{\theta_{4}})$,
\begin{align}\label{MWQ099}
|u(x,t)-u(y,s)|\leq C\big(|x-y|+|t-s|^{\frac{1}{p+\vartheta}}\big)^{\tilde{\alpha}},
\end{align}
for every $(x,t),(y,s)\in B_{1/4}\times(-1/2,0],$ where $\vartheta$ is defined by \eqref{VAR01}.
\end{theorem}

\begin{remark}
For the general time-space domain $\Omega_{T}$ containing the origin, define $R_{0}:=\min\{T^{\frac{1}{p+\vartheta}},\mathrm{dist}(0,\partial\Omega)\}$. By slightly modifying the proofs of Theorems \ref{ZWTHM90} and \ref{THM060}, we derive that the results in \eqref{QNAW001ZW} and \eqref{MWQ099} hold with $(-1/2,0]$, $B_{1/2}\times(-1/2,t_{0}]$ and $B_{1/4}\times(-1/2,0]$ replaced by $(-R^{p+\vartheta}_{0}/2,0]$, $B_{R_{0}/2}\times(-R_{0}^{p+\vartheta}/2,t_{0}]$ and $B_{R_{0}/4}\times(-R^{p+\vartheta}_{0}/2,0]$, respectively. In this case the universal constant $C$ will depend on $R_{0}$, but the exponents $\alpha$ and $\tilde{\alpha}$ are independent of $R_{0}$.

\end{remark}

\begin{remark}
Let $\nu$ be the unit outward normal defined on $\partial\Omega$. When the Dirichlet condition of $u=\psi\;\mathrm{on}\;\partial_{pa}\Omega_{T}$ is replaced with the Neumann condition as follows:
\begin{align*}
\mathbf{a}(x,t,u,\nabla u)\cdot\nu=0\;\mathrm{on}\;\partial\Omega\times(-T,0],\quad\mathrm{and}\;\,u=\psi\;\mathrm{on}\;\Omega\times\{t=-T\},
\end{align*}
the results in Theorems \ref{ZWTHM90} and \ref{THM060} also hold by the same arguments.
\end{remark}

%\begin{remark}
%Since our results cover the case of $p=2$, it doesn't need to analyze the stability for $p$ near $2$ any more.
%
%\end{remark}

%Especially our proof holds the promise of a wide application to more types of quasilinear parabolic equations.

The rest of the paper is organized as follows. Section \ref{SEC002} is to do some preliminary work, especially including the establishment of improved energy inequalities in Lemma \ref{lem003}. In Section \ref{SEC03} we complete the proofs of Theorems \ref{ZWTHM90} and \ref{THM060} by establishing three main ingredients required in the implementation of the De Giorgi truncation method. We now conclude the introduction by reviewing some previous relevant investigations.

Note that the singular and degenerate natures of equation \eqref{PO001} are the same to the following weighted parabolic $p$-Laplace equation
\begin{align*}
w_{1}\partial_{t}u-\mathrm{div}(w_{2}|\nabla u|^{p-2}\nabla u)=0.
\end{align*}
If $w_{1}=w_{2}=1$, it turns into the classical parabolic $p$-Laplace equation. In this case its modulus of ellipticity is $|\nabla u|^{p-2}$, which indicates that on the set $\{|\nabla u|=0\}$, the equation is singular for $1<p<2$, while it is degenerate for $p>2$. These singular and degenerate features do great damage to the regularity of the solution and pull it down to be only of $C^{1,\alpha}$ for some $0<\alpha<1$ compared to smooth solutions in the linear case of $p=2$. There is a long list of literature devoted to studying the smoothness of solutions to unweighted quasilinear elliptic and parabolic equations, see e.g. \cite{L2019,D1986,CD1988,D1983,D1993,U2008,E1982,DK1992,L1994,DGV2012,DGV2008,DGV200802,AL1983,DF198501,DF198502,DF1984,C1991,BSS2022} and the references therein. Besides its own mathematical interest, this nonlinear evolution model is also of great relevance to applications in nonlinear porous media flow, image processing and game theory etc, see e.g. \cite{DLK2013,ACM2004,BV2004,PS2008}. Moreover, based on the diffusion phenomena of the flows, we could classify the equation into the following three types. When $1<p<2$, the equation is called fast diffusion equation and its solution has a finite extinction time. By comparison, it is termed slow diffusion equation in the case of $p>2$ and its solution decays in time to the stable state at the rate of power-function type. The equation in the borderline case of $p=2$ becomes the linear heat equation, whose solution decays exponentially in the time variable. In addition, the regularity results have also been extended to general doubly nonlinear parabolic equations, see e.g. \cite{I1989,I199401,I199402,I1995,BDMS2018,BDL2021,BDGLS202301,BDGLS202302,KSU2012,LS2022,PV1993,VV2022,BV2010,BFV2018,V2007,DK2007,DKV1991,JX2019,JX2022,JRX2023,K1988}. Particularly in \cite{BDGLS202302} there is a systematic introduction on this topic.

For the weighted equation, it can be used to describe singular or degenerate diffusion phenomena occurring in inhomogeneous media as shown in \cite{RK1982,KR1981,KR1982}. The study on local regularity for the weighted equations originated from the famous work \cite{FKS1982}, where Fabes, Kenig and Serapioni \cite{FKS1982} established local H\"{o}lder estimates of solutions to second-order divergence form elliptic equations with the weight $|x|^{\theta_{2}}$ for any $\theta_{2}>-n$. After that, it brought a long series of papers involving different techniques and methods to study local behavior of solutions for all kinds of weighted elliptic and parabolic equations, see \cite{FP2013,DP2023,DPT2023,DPT202302,MZ202301,JX2023,STV2021,STV202102,GW1991,GW1990,BS2023,S2010,BS2019,BS1999,AA2004} and the references therein.

Note that the anisotropic weights considered in this paper is not translation invariant, which implies that the scaling properties of the weighted equation change greatly near the singular or degenerate points of the weights. This fact leads to that the solution usually becomes worse around those points and its regularity always falls to be only of $C^{\alpha}$, except for some solutions of special structures such as even solutions revealed in \cite{STV2021}. Since the even solutions of \cite{STV2021} are found in the presence of weighted linear elliptic operators, it is naturally to ask whether there exist similar $C^{1,\alpha}$ solutions to the weighted quasi-linear elliptic and parabolic equations under certain special structure condition. There makes no progress in pursuing a precise description for the damage effect of the weights on the regularity of solution until Dong, Li and Yang \cite{DLY2021,DLY2022} used spherical harmonic expansion to capture the sharp decay rate of solutions to the weighted elliptic equations near the degenerate point of the Caffarelli-Kohn-Nirenberg weight, see Lemma 2.2 in \cite{DLY2021} and Lemmas 2.2 and 5.1 in \cite{DLY2022} for finer details. Furthermore, these sharp exponents were utilized to solve the optimal gradient blow-up rates for the insulated conductivity problem arising from high-contrast composite materials, which has been previously considered as a challenging problem.

In addition, the considered anisotropic weights also have a significant application in achieving a complete classification (see \cite{LY202100}) for all $(-1)$-homogeneous axisymmetric no-swirl solutions to the stationary Navier-Stokes equations found in \cite{LLY201801,LLY201802}. See \cite{LY2023,MZ202301,MZ202302} for more properties and applications related to this type of anisotropic weights. Especially in \cite{MZ202302} Miao and Zhao studied a class of anisotropic Muckenhoupt weights possessing more general form of $|x'|^{\theta_{1}}|x|^{\theta_{2}}|x_{n}|^{\theta_{3}}$. A couple of intriguing questions such as the problem of establishing the corresponding weighted interpolation and Poincar\'{e} inequalities naturally arises from their investigation. One of the prime motivations to study such type of weights comes from the important role of the weight $|x_{n}|^{\theta_{3}}$ played in the establishment of the global regularity for solutions to fast diffusion equations as shown by Jin and Xiong \cite{JX2019,JX2022}. In particular, their results answered an open question asked by Berryman and Holland \cite{BH1980}. Similar results were also extended to porous medium equations in recent work \cite{JRX2023}.

\section{Preliminary}\label{SEC002}
For later use, we first fix some notations. For $x_{0}\in\mathbb{R}^{n}$, $t_{0}\in\mathbb{R}$ and $\rho,\tau>0$, let $B_{\rho}(x_{0})$ be the ball of centre $x_{0}$ and radius $\rho$. Define the backward parabolic cylinder as follows:
\begin{align*}
[(x_{0},t_{0})+Q(\rho,\tau)]:=B_{\rho}(x_{0})\times(t_{0}-\tau,t_{0}].
\end{align*}
For simplicity, let $B_{\rho}:=B_{\rho}(0)$ and $Q(\rho,\tau):=[(0,0)+Q(\rho,\tau)]$. For $k\in\mathbb{R}$ and $u\in C((-1,0];L^{2}(B_{1},w_{1}))\cap L^{p}((-1,0);W^{1,p}(B_{1},w_{2}))$, define
\begin{align*}
(u-k)_{+}=\max\{u-k,0\},\quad(u-k)_{-}=\max\{k-u,0\}.
\end{align*}
For $E\subset B_{1}$ and $\widetilde{E}\subset B_{1}\times(-1,0]$, let
\begin{align}\label{MEA01}
|E|_{\mu_{w_{i}}}=\int_{E}w_{i}dx,\quad |\widetilde{E}|_{\nu_{w_{i}}}=\int_{\widetilde{E}}w_{i}dxdt,\quad i=1,2.
\end{align}
We start by recalling the following two measure lemmas established in previous work \cite{MZ202301,MZ202303}.
\begin{lemma}[see Lemma 2.1 in \cite{MZ202301}]\label{LEM860}
$d\mu:=|x'|^{\theta_{1}}|x|^{\theta_{2}}dx$ is a Radon measure if $(\theta_{1},\theta_{2})\in\mathcal{A}\cup\mathcal{B}$. Furthermore, there exists some constant $C=C(n,\theta_{1},\theta_{2})>0$ such that $C^{-1}R^{n+\theta_{1}+\theta_{2}}\leq\mu(B_{R})\leq C R^{n+\theta_{1}+\theta_{2}}$ for any $R>0$.

\end{lemma}

Based on Lemma \ref{LEM860} and using the proof of Lemma 2.3 in \cite{MZ202303} with a slight modification, we have the following switch lemma from the measures $w_{1}dxdt$ to $w_{2}dxdt.$
\begin{lemma}[see Lemma 2.3 in \cite{MZ202303}]\label{lem09}
Let $(\theta_{1},\theta_{2}),(\theta_{3},\theta_{4})\in\mathcal{A}\cup\mathcal{B}$ and $\theta_{3}+\min\{0,\theta_{4}\}>1-n$. Then there exists some constant $C_{0}=C_{0}(n,p,\theta_{1},\theta_{2},\theta_{3},\theta_{4})>0$ such that for any $\varepsilon,\rho\in(0,1/2]$, $m,\tau\in(0,\infty)$ and $\widetilde{E}\subset Q(m\rho,\tau\rho^{p+\vartheta})$, if
\begin{align*}
\frac{|\widetilde{E}|_{\nu_{w_{1}}}}{|Q(m\rho,\tau\rho^{p+\vartheta})|_{\nu_{w_{1}}}}\leq\varepsilon^{\beta},
\end{align*}
then
\begin{align*}
\frac{|\widetilde{E}|_{\nu_{w_{2}}}}{|Q(m\rho,\tau\rho^{p+\vartheta})|_{\nu_{w_{2}}}}\leq C_{0}\varepsilon^{n-1+\theta_{3}+\min\{0,\theta_{4}\}},
\end{align*}
where $\beta=n-1+\theta_{3}+\min\{0,\theta_{4}\}+\max\{0,\theta_{1}-\theta_{3}\}+\max\{0,\theta_{2}-\theta_{4}\}$, the measures $\nu_{w_{i}}$, $i=1,2$ are defined by \eqref{MEA01}.

\end{lemma}

For any $f\in C((-1,0];L^{p}(B_{1},w_{1}))\cap L^{r}((-1,0);L^{q}(B_{1},w_{2}))$ with $r,p,q>1$, we introduce the Steklov averages as follows: for $0<h\ll1$,
\begin{align*}
f_{h}(x,t)=
\begin{cases}
\frac{1}{h}\int^{t}_{t-h}f(x,s)ds,& t\in(-1+h,0),\\
0,&t\leq-1+h,
\end{cases}
\end{align*}
and
\begin{align*}
f_{\bar{h}}(x,t)=
\begin{cases}
\frac{1}{h}\int^{t+h}_{t}f(x,s)ds,& t\in(-1,-h),\\
0,&t\geq-h.
\end{cases}
\end{align*}
We are now ready to establish the following improved energy inequalities.
\begin{lemma}[Improved energy inequalities]\label{lem003}
Assume as in Theorems \ref{ZWTHM90} and \ref{THM060}. Suppose that $u$ is the solution to problem \eqref{PO001} with $\Omega_{T}=B_{1}\times(-1,0]$. Let $v_{\pm}:=(u-k)_{\pm}$ with $k\in\mathbb{R}$ and $[(x_{0},t_{0})+Q(\rho,\tau)]\subset B_{1}\times(-1,0]$. Then we obtain that

$(i)$ for any $\xi\in C^{\infty}([(x_{0},t_{0})+Q(\rho,\tau)])$ which vanishes on $\partial B_{\rho}(x_{0})\times(t_{0}-\tau,t_{0})$ and satisfies that $0\leq\xi\leq1$, and any $s\in(t_{0}-\tau,t_{0}),$
\begin{align}\label{ENE01}
&\int_{B_{\rho}(x_{0})}v_{\pm}^{2}\xi^{p}(x,s)w_{1}dx+\frac{\lambda_{1}}{3}\int_{B_{\rho}(x_{0})\times(t_{0}-\tau,s)}|\nabla(v_{\pm}\xi)|^{p}w_{2}dxdt\notag\\
&\leq\int_{B_{\rho}(x_{0})}v_{\pm}^{2}\xi^{p}(x,t_{0}-\tau)w_{1}dx\notag\\
&\quad+C\int_{[(x_{0},t_{0})+Q(\rho,\tau)]}\big(v_{\pm}^{2}|\partial_{t}\xi|w_{1}+v_{\pm}^{p}|\nabla\xi|^{p}w_{2}\big)dxdt\notag\\
&\quad+C\|\phi\|_{L^{l_{0}}(B_{1}\times(-1,0),w_{2})}|[(x_{0},t_{0})+Q(\rho,\tau)]\cap \{v_{\pm}>0\}|_{\nu_{w_{2}}}^{1-\frac{1}{l_{0}}};
\end{align}

$(ii)$ for any $\zeta\in C_{0}^{\infty}(B_{\rho}(x_{0}))$ with $0\leq\zeta\leq1$,
\begin{align}\label{ENE02}
&\int_{[(x_{0},t_{0})+Q(\rho,\tau)]}|\nabla(v_{\pm}\zeta)|^{p}w_{2}dxdt\notag\\
&\leq C\sup_{[(x_{0},t_{0})+Q(\rho,\tau)]}v_{\pm}^{p}\int_{[(x_{0},t_{0})+Q(\rho,\tau)]\cap \{v_{\pm}>0\}}(|\nabla\zeta|^{p}+\zeta^{p})w_{2}dxdt\notag\\
&\quad+C\|\phi\|_{L^{l_{0}}(B_{1}\times(-1,0),w_{2})}|[(x_{0},t_{0})+Q(\rho,\tau)]\cap \{v_{\pm}>0\}|_{\nu_{w_{2}}}^{1-\frac{1}{l_{0}}},
\end{align}
where $\phi=\phi_{1}+\phi_{2}^{\frac{p}{p-1}}+\phi_{3}^{\frac{p}{p-1}}$.

\end{lemma}
\begin{remark}\label{REM09}
By adding the restricted condition in \eqref{CON01}, we derive the improved upper bound in \eqref{ENE01} by getting rid of the term $|\xi|^{p}$ in Lemma 3.1 of \cite{MZ202303}, which is a necessary condition for the establishment of global $L^{\infty}$ estimate in Theorem \ref{CORO06}. Moreover, we see from \eqref{MQ98} below that the assumed condition of $\lambda_{1}\gg\lambda_{3}$ in \eqref{CON01} can be quantitatively described as follows:
\begin{align*}
\lambda_{1}\geq C\lambda_{3},\quad\text{for some large constant }C=C(n,p,\theta_{3},\theta_{4})>0.
\end{align*}
 
\end{remark}

\begin{remark}
For any $\zeta\in C^{\infty}_{0}(B_{\rho}(x_{0}))$ with $0\leq\zeta\leq1$, we achieve an important improvement in \eqref{ENE02} by removing the term $\int_{B_{\rho}(x_{0})}(v_{\pm}^{2}(x,t_{0}-\tau)-v_{\pm}^{2}(x,s))\zeta^{p}(x)w_{1}dx$ in \eqref{ENE01} with any $s\in(t_{0}-\tau,t_{0})$, which is critical to the establishment of the decay estimates in Lemma \ref{lem005} below.

\end{remark}

\begin{remark}\label{RE06}
Let the considered domain $[(x_{0},t_{0})+Q(\rho,\tau)]$ become $Q(1,1)$ and choose $(\pm u-k)_{+}$ with $k\geq\|\psi\|_{L^{\infty}(\partial_{pa}Q(1,1))}$. Then we have $(\pm u-k)_{+}=0$ on $\partial_{pa}Q(1,1)$, which implies that \eqref{ENE01} holds for any $\xi\in C^{\infty}(Q(1,1))$ which may not vanish on $\partial B_{1}\times(-1,0)$.
\end{remark}

\begin{proof}[Proof of Lemma \ref{lem003}]
Assume without loss of generality that $(x_{0},t_{0})=(0,0)$.

{\bf Step 1.} In order to prove \eqref{ENE01}, we see from the proof of Lemma 3.1 in \cite{MZ202303} that it suffices to demonstrate that for any $-\tau\leq s\leq0$,
\begin{align}\label{O90}
&\int_{-\tau}^{s}\int_{B_{\rho}}\pm b(x,t,u,\nabla u)(u-k)_{\pm}\xi^{p}w_{2}\notag\\
&\leq \frac{\lambda_{1}}{6}\int_{B_{\rho}\times(-\tau,s)}|\nabla(u-k)_{\pm}|^{p}\xi^{p}w_{2}+\frac{\lambda_{1}}{12}\int_{B_{\rho}\times(-\tau,s)}(u-k)_{\pm}^{p}|\nabla\xi|^{p}w_{2}\notag\\
&\quad+C\int_{(B_{\rho}\times(-\tau,s))\cap\{(u-k)_{\pm}>0\}}\phi_{3}^{\frac{p}{p-1}}w_{2}.
\end{align}
Observe from $\mathrm{(}\mathbf{H3}\mathrm{)}$ and Young's inequality that
\begin{align}\label{DE90}
&\int_{-\tau}^{s}\int_{B_{\rho}}\pm b(x,t,u,\nabla u)(u-k)_{\pm}\xi^{p}w_{2}\notag\\
&\leq\lambda_{3}\int_{B_{\rho}\times(-\tau,s)}|\nabla(u-k)_{\pm}|^{p-1}(u-k)_{\pm}\xi^{p}w_{2}+\int_{B_{\rho}\times(-\tau,s)}(u-k)_{\pm}\phi_{3}\xi^{p}w_{2}\notag\\
&\leq \frac{\lambda_{1}}{12}\int_{B_{\rho}\times(-\tau,s)}|\nabla(u-k)_{\pm}|^{p}\xi^{p}w_{2}+\frac{C_{1}\lambda_{3}^{p}}{\lambda_{1}^{p-1}}\int_{B_{\rho}\times(-\tau,s)}(u-k)_{\pm}^{p}\xi^{p}w_{2}\notag\\
&\quad+C\int_{(B_{\rho}\times(-\tau,s))\cap\{(u-k)_{\pm}>0\}}\phi_{3}^{\frac{p}{p-1}}w_{2},
\end{align}
where $C_{1}=C_{1}(p)$. According to the assumed condition of $\lambda_{1}\gg\lambda_{3}$ in \eqref{CON01} and using Theorem 15.23 in \cite{HKM2006}, we obtain
\begin{align}\label{MQ98}
&\frac{C_{1}\lambda_{3}^{p}}{\lambda_{1}^{p-1}}\int_{B_{\rho}\times(-\tau,s)}((u-k)_{\pm}\xi)^{p}w_{2}dxdt\notag\\
&\leq \frac{C_{2}\lambda_{3}^{p}}{\lambda_{1}^{p-1}}\int_{B_{\rho}\times(-\tau,s)}|\nabla((u-k)_{\pm}\xi)|^{p}w_{2}dxdt\notag\\
&\leq\frac{2^{p-1}C_{2}\lambda_{3}^{p}}{\lambda_{1}^{p-1}}\int_{B_{\rho}\times(-\tau,s)}\big(|\nabla(u-k)_{\pm}|^{p}\xi^{p}+(u-k)_{\pm}^{p}|\nabla\xi|^{p}\big)w_{2}dxdt\notag\\
&\leq\frac{\lambda_{1}}{12}\int_{B_{\rho}\times(-\tau,s)}\big(|\nabla(u-k)_{\pm}|^{p}\xi^{p}+(u-k)_{\pm}^{p}|\nabla\xi|^{p}\big)w_{2}dxdt,
\end{align}
where $C_{2}=C_{2}(n,p,\theta_{3},\theta_{4})$. Inserting this into \eqref{DE90}, we obtain that \eqref{O90} holds. Then applying \eqref{O90} to the proof of Lemma 3.1 in \cite{MZ202303}, we see that \eqref{ENE01} holds.

{\bf Step 2.} It remains to prove \eqref{ENE02}. For $0<h\ll1$, let $\varphi=\mp\mathbbm{1}_{\{(u_{h}-k)_{\pm}>0\}}\zeta^{p}$,
where
\begin{align*}
\mathbbm{1}_{\{(u_{h}-k)_{\pm}>0\}}(x,t):=
\begin{cases}
1,&\text{if }(x,t)\in\{(u_{h}-k)_{\pm}>0\},\\
0,&\text{otherwise}.
\end{cases}
\end{align*}
Choosing the test function $\varphi_{\bar{h}}$, we obtain that for any $-\tau\leq s\leq 0$,
\begin{align}\label{WM01}
&\int_{B_{\rho}}u\varphi_{\bar{h}} w_{1}dx\Big|_{-\tau}^{s}+\int_{-\tau}^{s}\int_{B_{\rho}}(-u\partial_{t}\varphi_{\bar{h}} w_{1}+w_{2}\mathbf{a}(x,t,u,\nabla u)\cdot\nabla\varphi_{\bar{h}})dxdt\notag\\
&=\int_{-\tau}^{s}\int_{B_{\rho}}b(x,t,u,\nabla u)\varphi_{\bar{h}} w_{2}dxdt,
\end{align}
which leads to that
\begin{align*}
&\int_{-\tau}^{s}\int_{B_{\rho}}(\partial_{t}u_{h}\varphi w_{1}+w_{2}[\mathbf{a}(x,t,u,\nabla u)]_{h}\cdot\nabla\varphi)dxdt\notag\\
&=\int_{-\tau}^{s}\int_{B_{\rho}}[b(x,t,u,\nabla u)]_{h}\varphi w_{2}dxdt.
\end{align*}
Then integrating by parts and using Lemma 2.2 in \cite{MZ202303}, we obtain
\begin{align*}
&\int_{-\tau}^{s}\int_{B_{\rho}}\partial_{t}u_{h}\varphi w_{1}dxdt=-\int_{-\tau}^{s}\int_{B_{\rho}}\partial_{t}(u_{h}-k)_{\pm}\zeta^{p}w_{1}dxdt\notag\\
&=\int_{B_{\rho}}((u_{h}-k)_{\pm}(x,-\tau)-(u_{h}-k)_{\pm}(x,s))\zeta^{p}(x)w_{1}dx\notag\\
%&\quad+p\int_{B_{\rho}\times(-\tau,s)}(u_{h}-k)_{\pm}\xi^{p-1}\partial_{t}\xi w_{1}dxdt\notag\\
&\rightarrow\int_{B_{\rho}}((u-k)_{\pm}(x,-\tau)-(u-k)_{\pm}(x,s))\zeta^{p}(x)w_{1}dx,\quad\text{as }h\rightarrow0.
\end{align*}
As for the remaining two terms in \eqref{WM01}, by first letting $h\rightarrow0$ and then utilizing the structure conditions in $\mathrm{(}\mathbf{H1}\mathrm{)}$--$\mathrm{(}\mathbf{H3}\mathrm{)}$, it follows from Young's inequality and Lemma 2.2 that
\begin{align*}
&\int_{B_{\rho}\times(-\tau,s)}[\mathbf{a}(x,t,u,\nabla u)]_{h}\cdot\nabla\varphi w_{2}dxdt\notag\\
&\rightarrow\mp p\int_{(B_{\rho}\times(-\tau,s))\cap\{(u-k)_{\pm}>0\}}\zeta^{p-1}\mathbf{a}(x,t,u,\nabla u)\cdot\nabla\zeta w_{2}dxdt\notag\\
&\leq\lambda_{2}p\int_{B_{\rho}\times(-\tau,s)}|\nabla(u-k)_{\pm}|^{p-1}\zeta^{p-1}|\nabla\zeta|w_{2}dxdt\notag\\
&\quad+p\int_{(B_{\rho}\times(-\tau,s))\cap\{(u-k)_{\pm}>0\}}\phi_{2}\zeta^{p-1}|\nabla\zeta|w_{2}dxdt,
\end{align*}
and
\begin{align*}
&\int_{B_{\rho}\times(-\tau,s)}[b(x,t,u,\nabla u)]_{h}\varphi w_{2}dxdt\notag\\
&\rightarrow\mp\int_{(B_{\rho}\times(-\tau,s))\cap\{(u-k)_{\pm}>0\}} b(x,t,u,\nabla u)\zeta^{p}w_{2}dxdt\notag\\
&\leq\lambda_{3}\int_{B_{\rho}\times(-\tau,s)}|\nabla(u-k)_{\pm}|^{p-1}\zeta^{p}w_{2}+\int_{(B_{\rho}\times(-\tau,s))\cap\{(u-k)_{\pm}>0\}}\phi_{3}\zeta^{p}w_{2}.
\end{align*}
Therefore, a consequence of these above facts shows that for any $-\tau\leq s\leq0$,
\begin{align}\label{ENE03}
&\int_{B_{\rho}}(v_{\pm}(x,-\tau)-v_{\pm}(x,s))\zeta^{p}(x)w_{1}dx\notag\\
&\leq C\int_{Q(\rho,\tau)}|\zeta\nabla v_{\pm}|^{p-1}(|\nabla\zeta|+\zeta)w_{2}dxdt\notag\\
&\quad+C\int_{Q(\rho,\tau)\cap\{v_{\pm}>0\}}(\phi_{2}|\nabla\zeta|+\phi_{3}\zeta)\zeta^{p-1}w_{2}dxdt,
\end{align}
where $v_{\pm}=(u-k)_{\pm}$. Making use of \eqref{ENE03}, we have from Young's inequality, H\"{o}lder's inequality and Lemma 2.1 in \cite{MZ202303} that for any $-\tau\leq s\leq0$,
\begin{align}\label{WE900}
&\int_{B_{\rho}}(v_{\pm}^{2}(x,-\tau)-v^{2}_{\pm}(x,s))\zeta^{p}(x)w_{1}dx\notag\\
&\leq 2\sup\limits_{Q(\rho,\tau)}v_{\pm}\int_{B_{\rho}}(v_{\pm}(x,-\tau)-v_{\pm}(x,s))\zeta^{p}(x)w_{1}dx\notag\\
&\leq C\sup\limits_{Q(\rho,\tau)}v_{\pm}\int_{Q(\rho,\tau)}|\zeta\nabla v_{\pm}|^{p-1}(|\nabla\zeta|+\zeta)w_{2}dxdt\notag\\
&\quad+C\sup\limits_{Q(\rho,\tau)}v_{\pm}\int_{Q(\rho,\tau)\cap\{v_{\pm}>0\}}(\phi_{2}|\nabla\zeta|+\phi_{3}\zeta)\zeta^{p-1}w_{2}dxdt\notag\\
&\leq \frac{\lambda_{1}}{12}\int_{Q(\rho,\tau)}|\zeta\nabla v_{\pm}|^{p}w_{2}dxdt+C\sup\limits_{Q(\rho,\tau)}v^{p}_{\pm}\int_{Q(\rho,\tau)\cap\{v_{\pm}>0\}}(|\nabla\zeta|^{p}+\zeta^{p})w_{2}dxdt\notag\\
&\quad+C\int_{Q(\rho,\tau)\cap\{v_{\pm}>0\}}(\phi_{2}^{\frac{p}{p-1}}+\phi_{3}^{\frac{p}{p-1}})w_{2}dxdt\notag\\
&\leq \frac{\lambda_{1}}{6}\int_{Q(\rho,\tau)}|\nabla(v_{\pm}\zeta)|^{p}w_{2}dxdt+C\sup\limits_{Q(\rho,\tau)}v^{p}_{\pm}\int_{Q(\rho,\tau)\cap\{v_{\pm}>0\}}(|\nabla\zeta|^{p}+\zeta^{p})w_{2}dxdt\notag\\
&\quad+C\|\phi\|_{L^{l_{0}}(B_{1}\times(-1,0),w_{2})}|Q(\rho,\tau)\cap \{v_{\pm}>0\}|_{\nu_{w_{2}}}^{1-\frac{1}{l_{0}}},
\end{align}
where $\phi=\phi_{1}+\phi_{2}^{\frac{p}{p-1}}+\phi_{3}^{\frac{p}{p-1}}.$ Therefore, a combination of \eqref{ENE01} and \eqref{WE900} leads to that
\begin{align*}
&\frac{\lambda_{1}}{3}\int_{Q(\rho,\tau)}|\nabla(v_{\pm}\zeta)|^{p}w_{2}dxdt\notag\\
&\leq\sup_{s\in(-\tau,0)}\int_{B_{\rho}}(v_{\pm}^{2}(x,-\tau)-v^{2}_{\pm}(x,s))\zeta^{p}(x)w_{1}dx+C\int_{Q(\rho,\tau)}v_{\pm}^{p}|\nabla\zeta|^{p}w_{2}dxdt\notag\\
&\quad+C\|\phi\|_{L^{l_{0}}(B_{1}\times(-1,0),w_{2})}|Q(\rho,\tau)\cap \{v_{\pm}>0\}|_{\nu_{w_{2}}}^{1-\frac{1}{l_{0}}}\notag\\
&\leq \frac{\lambda_{1}}{6}\int_{Q(\rho,\tau)}|\nabla(v_{\pm}\zeta)|^{p}w_{2}dxdt+C\sup\limits_{Q(\rho,\tau)}v^{p}_{\pm}\int_{Q(\rho,\tau)\cap\{v_{\pm}>0\}}(|\nabla\zeta|^{p}+\zeta^{p})w_{2}dxdt\notag\\
&\quad+C\|\phi\|_{L^{l_{0}}(B_{1}\times(-1,0),w_{2})}|Q(\rho,\tau)\cap \{v_{\pm}>0\}|_{\nu_{w_{2}}}^{1-\frac{1}{l_{0}}},
\end{align*}
which implies that \eqref{ENE02} holds.

\end{proof}

We next list anisotropically weighted isoperimetric inequality and Sobolev embedding theorem. These two tools are critical to the implementation for the De Giorgi truncation method.

\begin{lemma}[see Lemma 2.3 in \cite{MZ202303}]\label{prop002}
For $n\geq2$ and $1<p<\infty$, let $(\theta_{1},\theta_{2})\in[(\mathcal{A}\cup\mathcal{B})\cap(C_{p}\cup\mathcal{D}_{p})]\cup\{\theta_{1}=0,\,\theta_{2}\geq n(p-1)\}$. Then there exists some constant $1<q=q(n,p,\theta_{1},\theta_{2})<p$ such that for any $R>0$, $l>k$ and $u\in W^{1,q}(B_{R},w_{1})$,
\begin{align}\label{pro001}
&(l-k)^{q}\bigg(\int_{\{u\geq l\}\cap B_{R}}w_{1}dx\bigg)^{q}\int_{\{u\leq k\}\cap B_{R}}w_{1}dx\notag\\
&\leq C(n,p,\theta_{1},\theta_{2})R^{q(n+\theta_{1}+\theta_{2}+1)}\int_{\{k<u<l\}\cap B_{R}}|\nabla u|^{q}w_{1}dx,
\end{align}
and
\begin{align}\label{pro002}
&(l-k)^{q}\bigg(\int_{\{u\leq k\}\cap B_{R}}w_{1}dx\bigg)^{q}\int_{\{u\geq l\}\cap B_{R}}w_{1}dx\notag\\
&\leq C(n,p,\theta_{1},\theta_{2})R^{q(n+\theta_{1}+\theta_{2}+1)}\int_{\{k<u<l\}\cap B_{R}}|\nabla u|^{q}w_{1}dx,
\end{align}
where $w_{1}=|x'|^{\theta_{1}}|x|^{\theta_{2}}dx.$
\end{lemma}

We now establish the desired anisotropically weighted Sobolev inequality as follows.
\begin{prop}\label{prop001}
For $n\geq2$, $1<p\leq2$, $R>0$, $\theta_{1}+\theta_{2}>p-n$, let $u\in L^{2}(B_{R},w_{1})\cap W_{0}^{1,p}(B_{R},w_{2})$. Then
\begin{align*}
\int_{B_{R}}|u|^{p\tilde{\chi}}w_{2}dx\leq C(n,p,\theta_{1},\theta_{2})\int_{B_{R}}|\nabla u|^{p}w_{2}dx\left(\int_{B_{R}}|u|^{2}w_{1}dx\right)^{\chi-1},
\end{align*}
where $\chi$ and $\tilde{\chi}$ are given by
\begin{align}\label{chi}
\chi=\frac{n+p+\theta_{1}+\theta_{2}}{n+\theta_{1}+\theta_{2}},\quad \tilde{\chi}=\frac{n+\theta_{1}+\theta_{2}+2}{n+\theta_{1}+\theta_{2}}.
\end{align}

\end{prop}
\begin{remark}
It is worth emphasizing that the form of the above Sobolev inequality is distinct from that in Proposition 2.4 of \cite{MZ202303} and has advantages in simplifying the proof for the oscillation improvement of the solution in Lemma \ref{LEM0035} below by avoiding the similar variable substitution adopted in Lemma 4.8 of \cite{MZ202303}.

\end{remark}

\begin{proof}
Utilizing the anisotropic Caffarelli-Kohn-Nirenberg inequality in \cite{LY2023}, we derive that for any $u\in W_{0}^{1,p}(B_{R},w_{2})$,
\begin{align*}
&\left(\int_{B_{R}}|u|^{\frac{p(n+\theta_{1}+\theta_{2})}{n+\theta_{1}+\theta_{2}-p}}|x'|^{\frac{\theta_{3}(n+\theta_{1}+\theta_{2})-p\theta_{1}}{n+\theta_{1}+\theta_{2}-p}}|x|^{\frac{\theta_{4}(n+\theta_{1}+\theta_{2})-p\theta_{2}}{n+\theta_{1}+\theta_{2}-p}}dx\right)^{\frac{n+\theta_{1}+\theta_{2}-p}{n+\theta_{1}+\theta_{2}}}\notag\\
&\leq C\int_{B_{R}}|\nabla u|^{p}|x'|^{\theta_{3}}|x|^{\theta_{4}}dx.
\end{align*}
It then follows from H\"{o}lder's inequality that
\begin{align*}
&\int_{B_{R}}|u|^{p\tilde{\chi}}|x'|^{\theta_{3}}|x|^{\theta_{4}}dx\notag\\
&=\int_{B_{R}}|u|^{p}|x'|^{\theta_{3}-\theta_{1}(\chi-1)}|x|^{\theta_{4}-\theta_{2}(\chi-1)}|u|^{2(\chi-1)}|x'|^{\theta_{1}(\chi-1)}|x|^{\theta_{2}(\chi-1)}dx\notag\\
&\leq\left(\int_{B_{R}}|u|^{\frac{p}{2-\chi}}|x'|^{\frac{\theta_{3}-\theta_{1}(\chi-1)}{2-\chi}}|x|^{\frac{\theta_{4}-\theta_{2}(\chi-1)}{2-\chi}}dx\right)^{2-\chi}\left(\int_{B_{R}}|u|^{2}|x'|^{\theta_{1}}|x|^{\theta_{2}}dx\right)^{\chi-1}\notag\\
&\leq C\int_{B_{R}}|\nabla u|^{p}|x'|^{\theta_{3}}|x|^{\theta_{4}}dx\left(\int_{B_{R}}|u|^{2}|x'|^{\theta_{1}}|x|^{\theta_{2}}dx\right)^{\chi-1}.
\end{align*}
The proof is complete.

\end{proof}

A consequence of \eqref{ENE01}, Remark \ref{RE06} and Proposition \ref{prop001} gives the following global $L^{\infty}$ estimate.
\begin{theorem}\label{CORO06}
Assume as in Theorems \ref{ZWTHM90} and \ref{THM060}. Let $u$ be the weak solution of \eqref{PO001} with $\Omega_{T}=Q(1,1)$. Then we have
\begin{align*}
\|u\|_{L^{\infty}(Q(1,1))}\leq\|\psi\|_{L^{\infty}(\partial_{pa}Q(1,1))}+C\|\phi\|_{L^{l_{0}}(Q(1,1))}^{\frac{n+p+\theta_{1}+\theta_{2}}{p(n+\theta_{1}+\theta_{2}+2)}},
\end{align*}
where $\phi$ is given by \eqref{ZE90}.
\end{theorem}
\begin{proof}
For $i\geq0$, set $k_{i}=\|\psi\|_{L^{\infty}(\partial_{pa}Q(1,1))}+M-\frac{M}{2^{i}}$, where $M$ is to be determined later. By picking $\xi\equiv1$ in \eqref{ENE01}, we have from Remark \ref{RE06} that
\begin{align*}
&\sup\limits_{t\in(-1,0)}\int_{B_{1}}(u-k_{i})_{+}^{2}w_{1}dx+\frac{\lambda_{1}}{3}\int_{Q(1,1)}|\nabla(u-k)_{+}|^{p}w_{2}dxdt\notag\\
&\leq C\|\phi\|_{L^{l_{0}}(Q(1,1))}|Q(1,1)\cap\{u>k_{i}\}|_{\nu_{w_{2}}}^{1-\frac{1}{l_{0}}},
\end{align*}
which, in combination with Proposition \ref{prop001}, reads that
\begin{align*}
&\int_{Q(1,1)}|(u-k_{i})_{+}|^{p\tilde{\chi}}w_{2}dxdt\notag\\
&\leq C\int_{Q(1,1)}|\nabla(u-k)_{+}|^{p}w_{2}dxdt\bigg(\sup\limits_{t\in(-1,0)}\int_{B_{1}}(u-k_{i})_{+}^{2}w_{1}dx\bigg)^{\chi-1}\notag\\ &\leq C\|\phi\|^{\chi}_{L^{l_{0}}(Q(1,1))}|Q(1,1)\cap\{u>k_{i}\}|_{\nu_{w_{2}}}^{\frac{\chi(l_{0}-1)}{l_{0}}},
\end{align*}
where $\chi$ and $\tilde{\chi}$ are defined by \eqref{chi}. For brevity, write
$A_{i}:=Q(1,1)\cap\{u>k_{i}\}$. Hence we deduce
\begin{align*}
|A_{i+1}|_{\nu_{w_{2}}}\leq&\Bigg(\frac{C\|\phi\|_{L^{l_{0}}(Q(1,1))}2^{\frac{p\tilde{\chi}(i+1)}{\chi}}}{M^{\frac{p\tilde{\chi}}{\chi}}}\Bigg)^{\chi}|A_{i}|_{\nu_{w_{2}}}^{\frac{\chi(l_{0}-1)}{l_{0}}}\notag\\
\leq&\prod_{s=0}^{i}\left[\Bigg(\frac{C\|\phi\|_{L^{l_{0}}(Q(1,1))}2^{\frac{p\tilde{\chi}(i+1)}{\chi}}}{M^{\frac{p\tilde{\chi}}{\chi}}}\Bigg)^{\frac{l_{0}}{l_{0}-1}}\right]^{\big(\frac{\chi(l_{0}-1)}{l_{0}}\big)^{s+1}}|A_{0}|_{\nu_{w_{2}}}^{\big(\frac{\chi(l_{0}-1)}{l_{0}}\big)^{i+1}}\notag\\
\leq&\Bigg[\bigg(\frac{\overline{C}\|\phi\|_{L^{l_{0}}(Q(1,1))}}{M^{\frac{p\tilde{\chi}}{\chi}}}\bigg)^{\frac{l_{0}\chi}{\chi(l_{0}-1)-l_{0}}}|A_{0}|_{\nu_{w_{2}}}\Bigg]^{\big(\frac{\chi(l_{0}-1)}{l_{0}}\big)^{i+1}},
\end{align*}
where $\frac{\chi(l_{0}-1)}{l_{0}}>1$ in virtue of \eqref{E01}. Fix
\begin{align*}
M=\big(\overline{C}\|\phi\|_{L^{l_{0}}(Q(1,1))}\big)^{\frac{\chi}{p\tilde{\chi}}}(2|Q(1,1)|_{\nu_{w_{2}}})^{\frac{\chi(l_{0}-1)-l_{0}}{pl_{0}\tilde{\chi}}}.
\end{align*}
Then we obtain
\begin{align*}
|A_{i+1}|_{\nu_{w_{2}}}\leq2^{-\big(\frac{\chi(l_{0}-1)}{l_{0}}\big)^{i+1}}\rightarrow0,\quad\text{as }i\rightarrow\infty.
\end{align*}
This yields that
\begin{align*}
\sup_{Q(1,1)}u\leq\|\psi\|_{L^{\infty}(\partial_{pa}Q(1,1))}+C\|\phi\|_{L^{l_{0}}(Q(1,1))}^{\frac{\chi}{p\tilde{\chi}}}.
\end{align*}
Using the above arguments for the equation of $-u$, we finish the proof.
\end{proof}

\section{Local behavior for weak solutions }\label{SEC03}

Inspired by intrinsic scaling technique developed in \cite{D1993,CD1988,D1986}, in the following we need to find suitably rescaled cylinders whose dimensions accommodate the singularity and degeneracy of the weights and equations for the purpose of applying the De Giorgi truncation method to the considered weighted parabolic equations.

Set
\begin{align*}
\bar{\varepsilon}_{0}:=\frac{\varepsilon_{0}(2-p)}{p+\vartheta},
\end{align*}
where $\varepsilon_{0}$ is defined by \eqref{VAR01}. Since $1<p<2$, we see that for any $R\in(0,\frac{1}{2}]$ and $t_{0}\in[-\frac{1}{2},0]$,
\begin{align*}
[(0,t_{0})+Q(R^{1-\bar{\varepsilon}_{0}},R^{p+\vartheta})]\subset Q(1,1).
\end{align*}
Up to a translation, we assume that $(0,t_{0})=(0,0).$ Define
\begin{align*}
%\label{W08}
\mu^{+}=\sup\limits_{Q(R^{1-\bar{\varepsilon}_{0}},R^{p+\vartheta})}u,\quad\mu^{-}=\inf\limits_{Q(R^{1-\bar{\varepsilon}_{0}},R^{p+\vartheta})}u,
\end{align*}
and
\begin{align*}
%\label{W09}
\omega=\mathop{osc}\limits_{Q(R^{1-\bar{\varepsilon}_{0}},R^{p+\vartheta})}u=\mu^{+}-\mu^{-}.
\end{align*}
Introduce the cylinder as follows:
\begin{align}\label{Z999}
Q(a_{0}R,R^{p+\vartheta}),\quad a_{0}:=\Big(\frac{\sigma\omega}{A}\Big)^{\frac{p-2}{p+\vartheta}},
\end{align}
where $\sigma\in(0,1)$ and $A\in(1,\infty)$ are two positive constants to be chosen later, whose values depend only on the data. Remark that $a_{0}$ is called intrinsic scaling factor, since it makes the rescaled cylinders $Q(a_{0}R,R^{p+\vartheta})$ reflect the singular and degenerate natures exhibited by the weights and equations. In the unweighted case of $w_{1}=w_{2}=1$, the exponent $\frac{p-2}{p}$ of the scaling factor was revealed in previous work \cite{CD1988,D1993}. By contrast, the situation is quite different in the weighted case and the exponent becomes $\frac{p-2}{p+\vartheta}$. This is caused by the fact that the scaling factor is applied to the space-intrinsic geometry such that its form must match the stretching of the weights on the space.

Note that one of the following two inequalities must hold: either $\omega\leq \sigma^{-1}AR^{\varepsilon_{0}}$, or $\omega>\sigma^{-1}AR^{\varepsilon_{0}}$. If $\omega>\sigma^{-1}AR^{\varepsilon_{0}}$, then
\begin{align*}
Q(a_{0}R,R^{p+\vartheta})\subset Q(R^{1-\bar{\varepsilon}_{0}},R^{p+\vartheta}),\quad \mathop{osc}\limits_{Q(a_{0}R,R^{p+\vartheta})}u\leq\omega.
\end{align*}
For simplicity, let
\begin{align}\label{M01}
m:=\Big(\frac{\sigma\omega}{\mathcal{M}_{0}}\Big)^{\frac{p-2}{p+\vartheta}},\quad\mathcal{M}_{0}=2\|u\|_{L^{\infty}(Q(1,1))}+1,
\end{align}
which implies that $m\geq1$. When $A\geq\mathcal{M}_{0}$, we have $Q(mR,R^{p+\vartheta})\subset Q(a_{0}R,R^{p+\vartheta})$. Moreover, since $\frac{p+\vartheta}{2-p}>\varepsilon_{0}$ for $1<p<2$, we obtain that for any $R\in(0,1]$,
\begin{align}\label{M03}
mR\leq1,\quad\mathrm{if}\;\omega\geq\sigma^{-1}\mathcal{M}_{0}R^{\varepsilon_{0}}.
\end{align}

\subsection{Three main ingredients in the De Giorgi truncation method.}
We first improve the oscillation of $u$ in a small domain as follows.
\begin{lemma}\label{LEM0035}
Let $\sigma,m,\mathcal{M}_{0}$ be defined by \eqref{Z999}--\eqref{M01}. Then there exists a small constant $\gamma_{0}\in(0,1)$ depending only upon the above data and independent of $\sigma$ such that

$(i)$ if
\begin{align}\label{ZWZ007}
|Q(mR,R^{p+\vartheta})\cap\{u>\mu^{+}-\sigma\omega\}|_{\nu_{w_{1}}}\leq\gamma_{0}|Q(mR,R^{p+\vartheta})|_{\nu_{w_{1}}},
\end{align}
then we have either $\omega\leq \sigma^{-1}\mathcal{M}_{0}R^{\varepsilon_{0}}$, or
\begin{align}\label{DZ001}
u\leq \mu^{+}-\frac{\sigma\omega}{2},\quad\mathrm{in}\text{ $Q(mR/2,(R/2)^{p+\vartheta})$;}
\end{align}

$(ii)$ if
\begin{align*}
|Q(mR,R^{p+\vartheta})\cap\{u<\mu^{-}+\sigma\omega\}|_{\nu_{w_{1}}}\leq\gamma_{0}|Q(mR,R^{p+\vartheta})|_{\nu_{w_{1}}},
\end{align*}
then we have either $\omega\leq \sigma^{-1}\mathcal{M}_{0}R^{\varepsilon_{0}}$, or
\begin{align}\label{ZWZ009}
u\geq \mu^{-}+\frac{\sigma\omega}{2},\quad\mathrm{in}\text{ $Q(mR/2,(R/2)^{p+\vartheta})$.}
\end{align}

\end{lemma}
\begin{remark}
According to the roles of $\sigma$ and $\gamma_{0}$ played in Lemma \ref{LEM0035}, $\sigma$ and $\gamma_{0}$ can be, respectively, called the level parameter and oscillation improvement constant. In contrast to \cite{D1993,CD1988,D1986}, we have actually redefined the intrinsic scaling factors $a_{0}$ and $m$ in \eqref{Z999}--\eqref{M01} by adding the level parameter $\sigma$ for the purpose of finding a $\sigma$-independent oscillation improvement constant $\gamma_{0}$ in Lemma \ref{LEM0035}. This is critical to its application in achieving the desired oscillation improvement of $u$ in Corollary \ref{AMZW01} below.

\end{remark}

\begin{proof}
In the following we only give the proof of \eqref{DZ001}, since the proof of \eqref{ZWZ009} is almost the same.
For $i=0,1,2,...,$ let
\begin{align*}
r_{i}=\frac{R}{2}+\frac{R}{2^{i+1}},\quad k_{i}=\mu^{+}-\sigma\omega+\frac{\sigma\omega}{2}(1-2^{-i}).
\end{align*}
Take a smooth cutoff function $\xi_{i}\in C^{\infty}(Q(mr_{i},r_{i}^{p+\vartheta}))$ satisfying that
\begin{align*}
\begin{cases}
0\leq\xi_{i}\leq1,&\mathrm{in}\;Q(mr_{i},r_{i}^{p+\vartheta}),\\
\xi_{i}=1,&\mathrm{in}\;Q(mr_{i+1},r_{i+1}^{p+\vartheta}),\\
\xi_{i}=0,&\mathrm{on}\;\partial_{pa}Q(mr_{i},r_{i}^{p+\vartheta}),\\
|\nabla\xi_{i}|\leq\frac{2^{i+3}}{mR},\quad|\partial_{t}\xi_{i}|\leq\frac{2}{r_{i}^{p+\vartheta}-r_{i+1}^{p+\vartheta}}\leq\frac{C2^{i}}{R^{p+\vartheta}}.
\end{cases}
\end{align*}
From \eqref{M01}, we see
\begin{align}\label{Z90}
(\sigma\omega)^{2}m^{\vartheta}=(\sigma\omega)^{\frac{p(2+\vartheta)}{p+\vartheta}}\mathcal{M}_{0}^{\frac{\vartheta(2-p)}{p+\vartheta}}=\Big(\frac{\sigma\omega}{m}\Big)^{p}\mathcal{M}_{0}^{2-p}.
\end{align}
Since $1<p<2$, then $\xi_{i}^{2}\leq\xi_{i}^{p}$ in $Q(mr_{i},r_{i}^{p+\vartheta})$. This, together with Lemma \ref{LEM860}, \eqref{ENE01}, \eqref{M03} and \eqref{Z90}, gives that if $\omega>\sigma^{-1}\mathcal{M}_{0}R^{\varepsilon_{0}}$,
\begin{align*}
&\sup\limits_{t\in(-r_{i}^{p+\vartheta},0)}\int_{B_{mr_{i}}}|(u-k_{i})_{+}\xi_{i}|^{2}w_{1}dx+\frac{\lambda_{1}}{3}\int_{Q(mr_{i},r_{i}^{p+\vartheta})}|\nabla((u-k_{i})_{+}\xi_{i})|^{p}w_{2}dxdt\notag\\
&\leq\sup\limits_{t\in(-r_{i}^{p+\vartheta},0)}\int_{B_{mr_{i}}}(u-k_{i})^{2}_{+}\xi_{i}^{p}w_{1}dx+\frac{\lambda_{1}}{3}\int_{Q(mr_{i},r_{i}^{p+\vartheta})}|\nabla((u-k_{i})_{+}\xi_{i})|^{p}w_{2}dxdt\notag\\
&\leq C\int_{Q(mr_{i},r_{i}^{p+\vartheta})}\big((u-k_{i})_{+}^{2}|\partial_{t}\xi_{i}||x'|^{\theta_{1}-\theta_{3}}|x|^{\theta_{2}-\theta_{4}}+(u-k_{i})_{+}^{p}(|\nabla\xi_{i}|^{p}+1)\big)w_{2}dxdt\notag\\
&\quad+C|Q(mr_{i},r_{i}^{p+\vartheta})\cap \{u>k_{i}\}|_{\nu_{w_{2}}}^{1-\frac{1}{l_{0}}}\notag\\
&\leq C\bigg(\frac{2^{i}(\sigma\omega)^{2}m^{\vartheta}}{R^{p}}+2^{pi}\Big(\frac{\sigma\omega}{mR}\Big)^{p}\bigg)|Q(mr_{i},r_{i}^{p+\vartheta})\cap \{u>k_{i}\}|_{\nu_{w_{2}}}\notag\\
&\quad+C|Q(mr_{i},r_{i}^{p+\vartheta})\cap \{u>k_{i}\}|_{\nu_{w_{2}}}^{1-\frac{1}{l_{0}}}\notag\\
&\leq \frac{C2^{pi}(\sigma\omega)^{2}m^{\vartheta}}{R^{p}}|Q(mr_{i},r_{i}^{p+\vartheta})\cap \{u>k_{i}\}|_{\nu_{w_{2}}}+C|Q(mr_{i},r_{i}^{p+\vartheta})\cap \{u>k_{i}\}|_{\nu_{w_{2}}}^{1-\frac{1}{l_{0}}}\notag\\
&\leq C\bigg(\frac{2^{pi}(\sigma\omega)^{2}m^{\vartheta+\frac{n+\theta_{3}+\theta_{4}}{l_{0}}}}{R^{\frac{p(l_{0}-1)-n-\theta_{1}-\theta_{2}}{l_{0}}}}+1\bigg)|Q(mr_{i},r_{i}^{p+\vartheta})\cap \{u>k_{i}\}|_{\nu_{w_{2}}}^{1-\frac{1}{l_{0}}}\notag\\
&\leq\frac{C2^{pi}(\sigma\omega)^{2}m^{\vartheta+\frac{n+\theta_{3}+\theta_{4}}{l_{0}}}}{R^{\frac{p(l_{0}-1)-n-\theta_{1}-\theta_{2}}{l_{0}}}}|Q(mr_{i},r_{i}^{p+\vartheta})\cap \{u>k_{i}\}|_{\nu_{w_{2}}}^{1-\frac{1}{l_{0}}},
\end{align*}
which, together with Proposition \ref{prop001}, shows that
\begin{align}\label{ZE01}
&\Big(\frac{\sigma\omega}{2^{i+2}}\Big)^{p\tilde{\chi}}|Q(mr_{i+1},r_{i+1}^{p+\vartheta})\cap \{u>k_{i+1}\}|_{\nu_{w_{2}}}\notag\\
&=(k_{i+1}-k_{i})^{p\tilde{\chi}}|Q(mr_{i+1},r_{i+1}^{p+\vartheta})\cap \{u>k_{i+1}\}|_{\nu_{w_{2}}}\notag\\
&\leq\int_{Q(mr_{i},r_{i}^{p+\vartheta})}|(u-k_{i})_{+}\xi_{i}|^{p\tilde{\chi}}w_{2}\notag\\
&\leq C\bigg(\sup\limits_{t\in(-r_{i}^{p+\vartheta},0)}\int_{B_{mr_{i}}}|(u-k_{i})_{+}\xi_{i}|^{2}w_{1}dx\bigg)^{\chi-1}\int_{Q(mr_{i},r_{i}^{p+\vartheta})}|\nabla((u-k_{i})_{+}\xi_{i})|^{p}w_{2}\notag\\
&\leq\Bigg(\frac{C2^{pi}(\sigma\omega)^{2}m^{\vartheta+\frac{n+\theta_{3}+\theta_{4}}{l_{0}}}}{R^{\frac{p(l_{0}-1)-n-\theta_{1}-\theta_{2}}{l_{0}}}}|Q(mr_{i},r_{i}^{p+\vartheta})\cap \{u>k_{i}\}|_{\nu_{w_{2}}}^{1-\frac{1}{l_{0}}}\Bigg)^{\chi}.
\end{align}
Define
\begin{align*}
F_{i}=\frac{|Q(mr_{i},r_{i}^{p+\vartheta})\cap \{u>k_{i}\}|_{\nu_{w_{2}}}}{|Q(mR,R^{p+\vartheta})|_{\nu_{w_{2}}}},\quad\mathrm{for}\;i\geq0.
\end{align*}
A consequence of \eqref{chi} and \eqref{M01} shows that
\begin{align*}
(\sigma\omega)^{2-\frac{p\tilde{\chi}}{\chi}}=m^{-\frac{(p+\vartheta)(n+\theta_{1}+\theta_{2})}{n+p+\theta_{1}+\theta_{2}}}\mathcal{M}_{0}^{\frac{(2-p)(n+\theta_{1}+\theta_{2})}{n+p+\theta_{1}+\theta_{2}}}.
\end{align*}
Substituting this into \eqref{ZE01}, we derive
\begin{align*}
F_{i+1}\leq&\Big(C2^{pi(1+\frac{\tilde{\chi}}{\chi})}F_{i}^{1-\frac{1}{l_{0}}}\Big)^{\chi}\notag\\
\leq&\prod\limits^{i}_{s=0}\Big[\Big(C2^{p(1+\frac{\tilde{\chi}}{\chi})(i-s)}\Big)^{\chi}\Big]^{\big(\frac{\chi(l_{0}-1)}{l_{0}}\big)^{s}}F_{0}^{\left(\frac{\chi(l_{0}-1)}{l_{0}}\right)^{i+1}}\notag\\
\leq&C^{\chi\sum\limits^{i}_{s=0}\left(\frac{\chi(l_{0}-1)}{l_{0}}\right)^{s}}2^{p(\chi+\tilde{\chi})\sum\limits_{s=0}^{i}(i-s)\left(\frac{\chi(l_{0}-1)}{l_{0}}\right)^{s}}F_{0}^{\left(\frac{\chi(l_{0}-1)}{l_{0}}\right)^{i+1}}\notag\\
\leq&(\widehat{C}F_{0})^{\left(\frac{\chi(l_{0}-1)}{l_{0}}\right)^{i+1}},
\end{align*}
where the constant $\widehat{C}$ is independent of $\sigma$. Choose
\begin{align}\label{E990}
\gamma_{0}=\big(2C_{0}\widehat{C}\big)^{-\frac{\beta}{n-1+\theta_{3}+\min\{0,\theta_{4}\}}},
\end{align}
where $\beta$ and $C_{0}$ are given by Lemma \ref{lem09}. Therefore, we obtain that if \eqref{ZWZ007} holds with $\gamma_{0}$ given by \eqref{E990},
\begin{align*}
F_{i+1}\leq2^{-\frac{n-1+\theta_{3}+\min\{0,\theta_{4}\}}{\beta}\big(\frac{\chi(l_{0}-1)}{l_{0}}\big)^{i+1}}\rightarrow0,\quad\mathrm{as}\;i\rightarrow\infty.
\end{align*}
The proof is finished.

\end{proof}

The key to application of Lemma \ref{LEM0035} lies in establishing the following decay estimates for the distribution function of the solution $u$.
\begin{lemma}\label{lem005}
Let $\sigma,m,\mathcal{M}_{0}$ be given by \eqref{Z999}--\eqref{M01}. There exists a large constant $\overline{C}_{0}>1$ depending only upon the data and independent of $\sigma$ such that for any $\gamma,\tau\in(0,1]$ and $\delta\in(0,\frac{1}{2}]$,

$(i)$ if
\begin{align}\label{D09}
\frac{| B_{mR/2}\cap\{u(\cdot,t)>\mu^{+}-\delta\omega\}|_{\mu_{w_{1}}}}{|B_{mR/2}|_{\mu_{w_{1}}}}\leq1-\gamma,\quad\forall\;t\in[-\tau R^{p+\vartheta},0],
\end{align}
then for any $j\geq1$, we have either
\begin{align*}
\omega\leq\sigma^{-1}\max\big\{\mathcal{M}_{0},(\delta^{-1}2^{j})^{\frac{pl_{0}(p+\vartheta)}{pl_{0}(2+\vartheta)+(p-2)(n+\theta_{3}+\theta_{4})}}\big\}R^{\varepsilon_{0}},
\end{align*}
or
\begin{align}\label{DEC001}
\frac{|Q(mR/2,\tau(R/2)^{p+\vartheta})\cap\{u>\mu^{+}-\frac{\delta\omega}{2^{j}}\}|_{\nu_{w_{1}}}}{|Q(mR/2,\tau(R/2)^{p+\vartheta})|_{\nu_{w_{1}}}}\leq\frac{\overline{C}_{0}}{\sqrt[q]{\gamma}\sqrt[p]{\tau}j^{\frac{p-q}{pq}}};
\end{align}

$(ii)$ if
\begin{align}\label{D98}
\frac{|B_{mR/2}\cap\{u(\cdot,t)<\mu^{-}+\delta\omega\}|_{\mu_{w_{1}}}}{|B_{mR/2}|_{\mu_{w_{1}}}}\leq1-\gamma,\quad\forall\;t\in[-\tau R^{p+\vartheta},0],
\end{align}
then for any $j\geq1$, we have either
\begin{align*}
\omega\leq\sigma^{-1}\max\big\{\mathcal{M}_{0},(\delta^{-1}2^{j})^{\frac{pl_{0}(p+\vartheta)}{pl_{0}(2+\vartheta)+(p-2)(n+\theta_{3}+\theta_{4})}}\big\}R^{\varepsilon_{0}},
\end{align*}
or
\begin{align}\label{D99}
\frac{|Q(mR/2,\tau(R/2)^{p+\vartheta})\cap\{u<\mu^{-}+\frac{\delta\omega}{2^{j}}\}|_{\nu_{w_{1}}}}{|Q(mR/2,\tau(R/2)^{p+\vartheta})|_{\nu_{w_{1}}}}\leq\frac{\overline{C}_{0}}{\sqrt[q]{\gamma}\sqrt[p]{\tau}j^{\frac{p-q}{pq}}},
\end{align}
where $q=q(n,p,\theta_{1},\theta_{2})\in(1,p)$ is given by Lemma \ref{prop002}.
\end{lemma}

\begin{proof}
\noindent{\bf Step 1.}
For $i\geq0$, set $k_{i}=\mu^{+}-\frac{\delta\omega}{2^{i}}$ and
\begin{align*}
A_{i}(t)=B_{mR/2}\cap\{u(\cdot,t)>k_{i}\},\quad A_{i}=Q(mR/2,\tau(R/2)^{p+\vartheta})\cap\{u>k_{i}\}.
\end{align*}
In light of \eqref{D09}, it follows from Lemma \ref{LEM860} that
\begin{align*}
|B_{mR/2}\setminus A_{i}(t)|_{\mu_{w_{1}}}\geq\gamma|B_{mR/2}|_{\mu_{w_{1}}}\geq C(n,\theta_{1},\theta_{2})\gamma(mR)^{n+\theta_{1}+\theta_{2}},
\end{align*}
which, together with \eqref{pro001}, shows that
\begin{align*}
|A_{i+1}(t)|_{\mu_{w_{1}}}\leq\frac{C2^{i}}{\delta\omega\sqrt[q]{\gamma}}(mR)^{\frac{(n+\theta_{1}+\theta_{2})(q-1)}{q}+1}\bigg(\int_{A_{i}(t)\setminus A_{i+1}(t)}|\nabla u|^{q}w_{1}dx\bigg)^{\frac{1}{q}}.
\end{align*}
Integrating this from $-\tau R^{p+\vartheta}$ to $0$, we have from H\"{o}lder's inequality that
\begin{align*}
&|A_{i+1}|_{\nu_{w_{1}}}\leq\frac{C2^{i}(\tau R^{p+\vartheta})^{\frac{q-1}{q}}}{\delta\omega\sqrt[q]{\gamma}}(mR)^{\frac{(n+\theta_{1}+\theta_{2})(q-1)}{q}+1}\bigg(\int_{A_{i}\setminus A_{i+1}}|\nabla u|^{q}w_{1}dxdt\bigg)^{\frac{1}{q}}.
\end{align*}
Since $1<q<p$, it follows from H\"{o}lder's inequality again that
\begin{align*}
&\bigg(\int_{A_{i}\setminus A_{i+1}}|\nabla u|^{q}w_{1}dxdt\bigg)^{\frac{1}{q}}\notag\\
&\leq\bigg(\int_{A_{i}\setminus A_{i+1}}|\nabla u|^{p}w_{2}dxdt\bigg)^{\frac{1}{p}}\bigg(\int_{A_{i}\setminus A_{i+1}}|x'|^{\frac{p\theta_{1}-q\theta_{3}}{p-q}}|x|^{\frac{p\theta_{2}-q\theta_{4}}{p-q}}dxdt\bigg)^{\frac{p-q}{pq}}\notag\\
&\leq (mR)^{\frac{\vartheta}{p}}|A_{i}\setminus A_{i+1}|_{\nu_{w_{1}}}^{\frac{p-q}{pq}}\bigg(\int_{Q(mR/2,\tau(R/2)^{p+\vartheta})}|\nabla (u-k_{i})_{+}|^{p}w_{2}dxdt\bigg)^{\frac{1}{p}}.
\end{align*}
Take a cutoff function $\zeta\in C_{0}^{\infty}(B_{mR})$ satisfying that
\begin{align*}
\zeta=1\;\mathrm{in}\;B_{mR/2},\quad\mathrm{and}\;\,0\leq\zeta\leq1,\;|\nabla\zeta|\leq\frac{4}{mR}\;\,\mathrm{in}\;B_{mR}.
\end{align*}
From \eqref{M01}, we deduce
\begin{align}\label{ZW999}
\omega^{p}m^{n+\theta_{3}+\theta_{4}-p}=\Big(\frac{\sigma}{\mathcal{M}_{0}}\Big)^{\frac{(p-2)(n+\theta_{3}+\theta_{4}-p)}{p+\vartheta}}\omega^{\frac{p(2+\vartheta)+(p-2)(n+\theta_{3}+\theta_{4})}{p+\vartheta}}.
\end{align}
Observe that
$$\frac{(2-p)(pl_{0}-n-\theta_{3}-\theta_{4})}{pl_{0}(2+\vartheta)+(p-2)(n+\theta_{3}+\theta_{4})}\leq1.$$
This, together with Lemma \ref{LEM860}, \eqref{ENE02}, \eqref{M03} and \eqref{ZW999}, gives that if
\begin{align*}
\omega>&\sigma^{-1}\max\big\{\mathcal{M}_{0},(\delta^{-1}2^{i})^{\frac{pl_{0}(p+\vartheta)}{pl_{0}(2+\vartheta)+(p-2)(n+\theta_{3}+\theta_{4})}}\big\}R^{\varepsilon_{0}}\notag\\
\geq&\max\Big\{\sigma^{-1}\mathcal{M}_{0},\big((\delta^{-1}2^{i})^{pl_{0}(p+\vartheta)}\sigma^{(p-2)(pl_{0}-n-\theta_{3}-\theta_{4})}\big)^{\frac{1}{pl_{0}(2+\vartheta)+(p-2)(n+\theta_{3}+\theta_{4})}}\Big\}R^{\varepsilon_{0}},
\end{align*}
then
\begin{align*}
&\int_{Q(mR/2,\tau(R/2)^{p+\vartheta})}|\nabla (u-k_{i})_{+}|^{p}w_{2}\leq\int_{Q(mR,\tau(R/2)^{p+\vartheta})}|\nabla((u-k_{i})_{+}\zeta)|^{p}w_{2}\notag\\
&\leq C\sup_{Q(mR,\tau(R/2)^{p+\vartheta})}(u-k_{i})_{+}^{p}\int_{Q(mR,\tau(R/2)^{p+\vartheta})\cap\{u>k_{i}\}}(|\nabla\zeta|^{p}+\zeta^{p})w_{2}\notag\\
&\quad+C|Q(mR,\tau(R/2)^{p+\vartheta})\cap\{u>k_{i}\}|_{\nu_{w_{2}}}^{\frac{l_{0}-1}{l_{0}}}\notag\\
&\leq C\bigg(\frac{(\delta\omega)^{p}R^{n+\theta_{1}+\theta_{2}}}{2^{pi}m^{p-n-\theta_{3}-\theta_{4}}}+m^{\frac{(n+\theta_{3}+\theta_{4})(l_{0}-1)}{l_{0}}}R^{\frac{(n+p+\theta_{1}+\theta_{2})(l_{0}-1)}{l_{0}}}\bigg)\notag\\
&\leq\frac{C\delta^{p}}{2^{pi}}\Big(\frac{\sigma}{\mathcal{M}_{0}}\Big)^{\frac{(p-2)(n+\theta_{3}+\theta_{4}-p)}{p+\vartheta}}\omega^{\frac{p(2+\vartheta)+(p-2)(n+\theta_{3}+\theta_{4})}{p+\vartheta}}R^{n+\theta_{1}+\theta_{2}}\notag\\
&\quad+C\Big(\frac{\sigma\omega}{\mathcal{M}_{0}}\Big)^{\frac{(l_{0}-1)(p-2)(n+\theta_{3}+\theta_{4})}{l_{0}(p+\vartheta)}}R^{\frac{(n+p+\theta_{1}+\theta_{2})(l_{0}-1)}{l_{0}}}\notag\\
&\leq \frac{C(\delta\omega)^{p}R^{n+\theta_{1}+\theta_{2}}}{2^{pi}m^{p-n-\theta_{3}-\theta_{4}}}.
\end{align*}
A consequence of these above facts reads that
\begin{align*}
|A_{i+1}|_{\nu_{w_{1}}}\leq&\frac{C}{\sqrt[q]{\gamma}\sqrt[p]{\tau}}|A_{i}\setminus A_{i+1}|_{\nu_{w_{1}}}^{\frac{p-q}{pq}}|Q(mR/2,\tau(R/2)^{p+\vartheta})|_{\nu_{w_{1}}}^{\frac{pq+q-p}{pq}}.
\end{align*}
Then we obtain that for every $j>i\geq0$,
\begin{align*}
j|A_{j}|_{\nu_{w_{1}}}^{\frac{pq}{p-q}}\leq&\sum^{j-1}_{i=0}|A_{i+1}|_{\nu_{w_{1}}}^{\frac{pq}{p-q}}\leq\frac{C}{\gamma^{\frac{p}{p-q}}\tau^{\frac{q}{p-q}}}|Q(mR/2,\tau(R/2)^{p+\vartheta})|_{\nu_{w_{1}}}^{\frac{pq}{p-q}},
\end{align*}
which indicates that \eqref{DEC001} holds.

\noindent{\bf Step 2.}
Similarly as before, for $i\geq0$, denote $\tilde{k}_{i}=\mu^{-}+\frac{\delta\omega}{2^{i}}$ and
\begin{align*}
\tilde{A}_{i}(t)=B_{mR/2}\cap\{u(\cdot,t)<k_{i}\},\quad \tilde{A}_{i}=Q(mR/2,\tau(R/2)^{p+\vartheta})\cap\{u<k_{i}\}.
\end{align*}
Combining \eqref{D98} and Lemma \ref{LEM860}, we have
\begin{align*}
|B_{mR/2}\setminus \tilde{A}_{i}(t)|_{\mu_{w_{1}}}\geq\gamma|B_{mR/2}|_{\mu_{w_{1}}}\geq C(n,\theta_{1},\theta_{2})\gamma(mR)^{n+\theta_{1}+\theta_{2}}.
\end{align*}
Using \eqref{pro002} rather than \eqref{pro001}, we obtain
\begin{align*}
|\tilde{A}_{i+1}(t)|_{\mu_{w_{1}}}\leq\frac{C2^{i}}{\delta\omega\sqrt[q]{\gamma}}(mR)^{\frac{(n+\theta_{1}+\theta_{2})(q-1)}{q}+1}\bigg(\int_{\tilde{A}_{i}(t)\setminus \tilde{A}_{i+1}(t)}|\nabla u|^{q}w_{1}dx\bigg)^{\frac{1}{q}}.
\end{align*}
Then following the left proof of \eqref{DEC001} above, we obtain that \eqref{D99} holds.

\end{proof}

We now carry out the forward expansion in time regarding the distribution function of the solution.
\begin{lemma}\label{LEM006}
Let $\sigma,m,\mathcal{M}_{0}$ be given by \eqref{Z999}--\eqref{M01}. For any $\bar{\gamma}\in(0,1)$ and $\bar{\delta}\in(0,2^{-1}]$, there exists a small positive constant $0<\sigma_{0}<1$ depending only upon $\bar{\gamma},\bar{\delta}$ and the above data but independent of $\sigma$ such that for any $0<\sigma\leq\sigma_{0}$,

$(i)$ if
\begin{align}\label{WMDN001}
\frac{|B_{mR}\cap\{u(\cdot,-R^{p+\vartheta})>\mu^{+}-\bar{\delta}\omega\}|_{\mu_{w_{1}}}}{|B_{mR}|_{\mu_{w_{1}}}}\leq1-\bar{\gamma},
\end{align}
then we have either $\omega\leq \sigma^{-1}\mathcal{M}_{0}R^{\varepsilon_{0}}$, or
\begin{align}\label{DQAF001}
\frac{|B_{mR}\cap\{u(\cdot,t)>\mu^{+}-2^{-\kappa_{0}}\bar{\delta}\omega\}|_{\mu_{w_{1}}}}{|B_{mR}|_{\mu_{w_{1}}}}\leq1-\frac{\bar{\gamma}}{2},\;\,\text{for all }t\in[-R^{p+\vartheta},0];
\end{align}

$(ii)$ if
\begin{align*}
\frac{|B_{mR}\cap\{u(\cdot,-R^{p+\vartheta})<\mu^{-}+\bar{\delta}\omega\}|_{\mu_{w_{1}}}}{|B_{mR}|_{\mu_{w_{1}}}}\leq1-\bar{\gamma},
\end{align*}
then we have either $\omega\leq \sigma^{-1}\mathcal{M}_{0}R^{\varepsilon_{0}}$, or
\begin{align}\label{WZPMQ001}
\frac{|B_{mR}\cap\{u(\cdot,t)<\mu^{-}+2^{-\kappa_{0}}\bar{\delta}\omega\}|_{\mu_{w_{1}}}}{|B_{mR}|_{\mu_{w_{1}}}}\leq1-\frac{\bar{\gamma}}{2},\;\,\text{for all }t\in[-R^{p+\vartheta},0],
\end{align}
where $\kappa_{0}$ is given by
\begin{align}\label{K90}
\kappa_{0}:=-\frac{\ln\Big(1-\frac{1}{\sqrt{1+\frac{\bar{\gamma}}{4(1-\bar{\gamma})}}}\Big)}{\ln2}.
\end{align}

\end{lemma}

\begin{remark}
By contrast with the proof of Lemma 4.7 in \cite{MZ202301} for $p=2$, we here present a more simple proof for the expansion in time by using the power gain of $1<p<2$. Moreover, the value of the captured constant $\kappa_{0}$ is explicit, which is a very sharp result.

\end{remark}

\begin{proof}
Since the proof of \eqref{WZPMQ001} is almost the same to \eqref{DQAF001}, we only prove \eqref{DQAF001} in the following. For $\tau\in(0,1]$ and $k\in[\mu^{-},\mu^{+}]$, denote
\begin{align*}
\bar{t}:=-R^{p+\vartheta},\quad A_{m}^{\tau}(k,R):=(B_{mR}\times[\bar{t},\bar{t}+\tau R^{p+\vartheta}])\cap\{u>k\}.
\end{align*}
Pick a cut-off function $\zeta\in C^{\infty}_{0}(B_{mR})$ such that $\zeta=1$ in $B_{(1-\varrho)mR}$, and $0\leq\zeta\leq1$, $|\nabla\zeta|\leq\frac{2}{\varrho mR}$ in $B_{mR}$, where $\varrho\in(0,1)$ is to be determined later. For simplicity, write $v:=(u-(\mu^{+}-\bar{\delta}\omega))_{+}$. Note that $$\frac{(2-p)(pl_{0}-n-\theta_{3}-\theta_{4})}{pl_{0}(2+\vartheta)+(p-2)(n+\theta_{3}+\theta_{4})}\leq1,$$
which, in combination with \eqref{ENE01}, \eqref{M03} and \eqref{WMDN001}, shows that if $$\omega>\sigma^{-1}\mathcal{M}_{0}R^{\varepsilon_{0}}=\max\big\{\sigma^{-1}\mathcal{M}_{0},\sigma^{-\frac{(2-p)(pl_{0}-n-\theta_{3}-\theta_{4})}{pl_{0}(2+\vartheta)+(p-2)(n+\theta_{3}+\theta_{4})}}\big\}R^{\varepsilon_{0}},$$
then
\begin{align}\label{QNE089}
&\sup\limits_{t\in(\bar{t},\bar{t}+\tau R^{p+\vartheta})}\int_{B_{(1-\varrho)mR}}v^{2}w_{1}dx\notag\\
&\leq\int_{B_{mR}}v^{2}(x,\bar{t})\zeta^{p}(x)w_{1}dx+C\int_{B_{mR}\times[\bar{t},\bar{t}+\tau R^{p+\vartheta}]}v^{p}(|\nabla\zeta|^{p}+|\zeta|^{p})w_{2}dxdt\notag\\
&\quad+C\big|A^{\tau}_{m}(\mu^{+}-\bar{\delta}\omega,R)\big|_{\nu_{w_{2}}}^{1-\frac{1}{l_{0}}}\notag\\
&\leq(\bar{\delta}\omega)^{2}(1-\bar{\gamma})|B_{mR}|_{\mu_{w_{1}}}+\frac{C(\bar{\delta}\omega)^{p}}{(\varrho mR)^{p}}\big|A^{\tau}_{m}(\mu^{+}-\bar{\delta}\omega,R)\big|_{\nu_{w_{2}}}\notag\\
&\quad+\big|A^{\tau}_{m}(\mu^{+}-\bar{\delta}\omega,R)\big|_{\nu_{w_{2}}}^{1-\frac{1}{l_{0}}}\notag\\
&\leq(\bar{\delta}\omega)^{2}(1-\bar{\gamma})|B_{mR}|_{\mu_{w_{1}}}\notag\\
&\quad+\Bigg(\frac{C\sigma^{\frac{(p-2)(n+\theta_{3}+\theta_{4}-pl_{0})}{l_{0}(p+\vartheta)}}\omega^{\frac{pl_{0}(2+\vartheta)+(p-2)(n+\theta_{3}+\theta_{4})}{l_{0}(p+\vartheta)}}}{\varrho^{p}R^{\frac{p(l_{0}-1)-n-\theta_{1}-\theta_{2}}{l_{0}}}}+1\Bigg)\big|A^{\tau}_{m}(\mu^{+}-\bar{\delta}\omega,R)\big|_{\nu_{w_{2}}}^{1-\frac{1}{l_{0}}}\notag\\
&\leq(\bar{\delta}\omega)^{2}(1-\bar{\gamma})|B_{mR}|_{\mu_{w_{1}}}+\frac{C\omega^{p}m^{\frac{n+\theta_{3}+\theta_{4}}{l_{0}}-p}}{\varrho^{p}R^{\frac{p(l_{0}-1)-n-\theta_{1}-\theta_{2}}{l_{0}}}}\big|A^{\tau}_{m}(\mu^{+}-\bar{\delta}\omega,R)\big|_{\nu_{w_{2}}}^{1-\frac{1}{l_{0}}}\notag\\
&\leq(\bar{\delta}\omega)^{2}|B_{mR}|_{\mu_{w_{1}}}\left(1-\bar{\gamma}+\frac{C\sigma^{2-p}}{\varrho^{p}}\bigg(\frac{|A_{m}^{\tau}(\mu^{+}-\bar{\delta}\omega,R)|_{\nu_{w_{2}}}}{|Q(mR,R^{p+\vartheta})|_{\nu_{w_{2}}}}\bigg)^{1-\frac{1}{l_{0}}}\right).
\end{align}
Remark that the term $\sigma^{2-p}$, which comes from $\omega^{p-2}=\big(\frac{\mathcal{M}_{0}}{\sigma}\big)^{p-2}m^{p+\vartheta}$, is a small term in virtue of the power gain of $1<p<2$.

Set $\kappa_{0}>1$. Then for any $t\in[\bar{t},\bar{t}+\tau R^{p+\vartheta}]$, we have
\begin{align*}
&\int_{B_{(1-\varrho)mR}}v^{2}(x,t)w_{1}dx\notag\\
&\geq(\bar{\delta}\omega)^{2}(1-2^{-\kappa_{0}})^{2}|B_{(1-\varrho)mR}\cap\{u(\cdot,t)>\mu^{+}-2^{-\kappa_{0}}\bar{\delta}\omega\}|_{\mu_{w_{1}}},
\end{align*}
which, together with \eqref{QNE089}, shows that
\begin{align*}
&|B_{(1-\varrho)mR}\cap\{u(\cdot,t)>\mu^{+}-2^{-\kappa_{0}}\bar{\delta}\omega\}|_{\mu_{w_{1}}}\notag\\
&\leq\frac{|B_{mR}|_{\mu_{w_{1}}}}{(1-2^{-\kappa_{0}})^{2}}\left(1-\bar{\gamma}+\frac{C\sigma^{2-p}}{\varrho^{p}}\bigg(\frac{|A_{m}^{\tau}(\mu^{+}-\bar{\delta}\omega,R)|_{\nu_{w_{2}}}}{|Q(mR,R^{p+\vartheta})|_{\nu_{w_{2}}}}\bigg)^{1-\frac{1}{l_{0}}}\right).
\end{align*}
Note that $|B_{mR}\setminus B_{(1-\varrho)mR}|_{\mu_{w_{1}}}\leq C\varrho|B_{mR}|_{\mu_{w_{1}}}$. Pick
\begin{align*}
\varrho=\sigma^{\frac{2-p}{p+1}}\bigg(\frac{|A_{m}^{\tau}(\mu^{+}-\bar{\delta}\omega,R)|_{\nu_{w_{2}}}}{|Q(mR,R^{p+\vartheta})|_{\nu_{w_{2}}}}\bigg)^{\frac{l_{0}-1}{l_{0}(p+1)}}.
\end{align*}
Then we derive that for any $t\in[\bar{t},\bar{t}+\tau R^{p+\vartheta}]$,
\begin{align}\label{D80}
&\frac{|B_{mR}\cap\{u(\cdot,t)>\mu^{+}-2^{-\kappa_{0}}\bar{\delta}\omega\}|_{\mu_{w_{1}}}}{|B_{mR}|_{\mu_{w_{1}}}}\notag\\
&\leq\frac{1-\bar{\gamma}}{(1-2^{-\kappa_{0}})^{2}}+\frac{C}{\varrho^{p}}\left[\varrho^{p+1}+\sigma^{2-p}\bigg(\frac{|A_{m}^{\tau}(\mu^{+}-\bar{\delta}\omega,R)|_{\nu_{w_{2}}}}{|Q(mR,R^{p+\vartheta})|_{\nu_{w_{2}}}}\bigg)^{1-\frac{1}{l_{0}}}\right]\notag\\
&\leq\frac{1-\bar{\gamma}}{(1-2^{-\kappa_{0}})^{2}}+C_{\ast}\sigma^{\frac{2-p}{p+1}}\bigg(\frac{|A_{m}^{\tau}(\mu^{+}-\bar{\delta}\omega,R)|_{\nu_{w_{2}}}}{|Q(mR,R^{p+\vartheta})|_{\nu_{w_{2}}}}\bigg)^{\frac{l_{0}-1}{l_{0}(p+1)}}.
\end{align}
Observe that
\begin{align}\label{A32}
\frac{|A_{m}^{\tau}(\mu^{+}-\bar{\delta}\omega,R)|_{\nu_{w_{2}}}}{|Q(mR,R^{p+\vartheta})|_{\nu_{w_{2}}}}\leq\tau.
\end{align}
Due to the power gain of $1<p<2$, we obtain
\begin{align}\label{M900}
C_{\ast}\sigma^{\frac{2-p}{p+1}}\leq\frac{\bar{\gamma}}{4},\quad\text{if }0<\sigma\leq\sigma_{0}:=\Big(\frac{\bar{\gamma}}{4C_{\ast}}\Big)^{\frac{p+1}{2-p}}.
\end{align}
Making use of \eqref{A32}--\eqref{M900} and applying \eqref{D80} with
$$\tau=1,\quad \kappa_{0}=-\frac{\ln\Big(1-\frac{1}{\sqrt{1+\frac{\bar{\gamma}}{4(1-\bar{\gamma})}}}\Big)}{\ln2},$$
we deduce that for any $t\in[-R^{p+\vartheta},0]$,
\begin{align*}
&\frac{|B_{mR}\cap\{u(\cdot,t)>\mu^{+}-2^{-\kappa_{0}}\bar{\delta}\omega\}|_{\mu_{w_{1}}}}{|B_{mR}|_{\mu_{w_{1}}}}\leq\frac{1-\bar{\gamma}}{(1-2^{-\kappa_{0}})^{2}}+C_{\ast}\sigma^{\frac{2-p}{p+1}}\leq1-\frac{\bar{\gamma}}{2}.
\end{align*}
Therefore, \eqref{DQAF001} is proved.

\end{proof}

A combination of Lemmas \ref{LEM0035}, \ref{lem005} and \ref{LEM006} shows the following improvement on oscillation of $u$.
\begin{corollary}\label{AMZW01}
The constants $\sigma,A$ (and then $a_{0},m$) can be determined and there exists a constant $\kappa_{\ast}>1$, both depending only upon the above data, such that

$(i)$ if
\begin{align}\label{W008}
\frac{|B_{mR/2}\cap\{u(\cdot,-(R/2)^{p+\vartheta})>\mu^{+}-2^{-1}\omega\}|_{\mu_{w_{1}}}}{|B_{mR/2}|_{\mu_{w_{1}}}}\leq\frac{1}{2},
\end{align}
then we have either $\omega\leq \sigma^{-1}AR^{\varepsilon_{0}}=2^{\kappa_{\ast}}AR^{\varepsilon_{0}}$, or
\begin{align}\label{W009}
\sup\limits_{Q(mR/4,(R/4)^{p+\vartheta})}u\leq \mu^{+}-\frac{\omega}{2^{\kappa_{\ast}+1}};
\end{align}

$(ii)$ if
\begin{align*}
\frac{|B_{mR/2}\cap\{u(\cdot,-(R/2)^{p+\vartheta})<\mu^{-}+2^{-1}\omega\}|_{\mu_{w_{1}}}}{|B_{mR/2}|_{\mu_{w_{1}}}}\leq\frac{1}{2},
\end{align*}
then we have either $\omega\leq \sigma^{-1}AR^{\varepsilon_{0}}=2^{\kappa_{\ast}}AR^{\varepsilon_{0}}$, or
\begin{align}\label{W010}
\inf\limits_{Q(mR/4,(R/4)^{p+\vartheta})}u\geq \mu^{-}+\frac{\omega}{2^{\kappa_{\ast}+1}}.
\end{align}

\end{corollary}
\begin{proof}
In the following we only prove \eqref{W009}, since the proof of \eqref{W010} is the same and thus omitted. Utilizing \eqref{W008} and applying Lemma \ref{LEM006} with $\bar{\delta}=\bar{\gamma}=\frac{1}{2}$, we obtain that for any $0<\sigma\leq\sigma_{0}$, there holds either $\omega\leq\sigma^{-1}\mathcal{M}_{0}R^{\varepsilon_{0}}$, or
\begin{align}\label{W011}
\frac{|B_{mR/2}\cap\{u(\cdot,t)>\mu^{+}-\frac{\omega}{2^{\kappa_{0}+1}}\}|_{\mu_{w_{1}}}}{|B_{mR}|_{\mu_{w_{1}}}}\leq\frac{3}{4},\quad\text{for all }t\in[-(R/2)^{p+\vartheta},0],
\end{align}
where $\kappa_{0}:=-\frac{\ln(1-\frac{2}{\sqrt{5}})}{\ln2}$ is given by \eqref{K90} with $\bar{\gamma}=\frac{1}{2}$. Observe that there exists a large constant $j_{0}>3$ depending only on the data such that
\begin{align}\label{W012}
2^{-(\kappa_{0}+1+j_{0})}\leq\sigma_{0},\quad\overline{C}_{0}\sqrt[q]{4}j_{0}^{-\frac{p-q}{pq}}\leq\gamma_{0},
\end{align}
where $\gamma_{0}$, $\overline{C}_{0}$ and $\sigma_{0}$ are, respectively, given by Lemmas \ref{LEM0035}, \ref{lem005} and \ref{LEM006}. Therefore, in view of \eqref{W011}--\eqref{W012} and applying Lemma \ref{lem005} with $\delta=\frac{1}{2^{\kappa_{0}+1}}$, $\sigma=\frac{1}{2^{\kappa_{0}+1+j_{0}}}$, $\tau=1$, $\gamma=\frac{1}{4}$, and $j=j_{0}$, we have either $\omega\leq \sigma^{-1}AR^{\varepsilon_{0}}=2^{\kappa_{0}+1+j_{0}}AR^{\varepsilon_{0}}$ with $A$ given by $A:=\max\big\{\mathcal{M}_{0},2^{\frac{pl_{0}(\kappa_{0}+1+j_{0})(p+\vartheta)}{pl_{0}(2+\vartheta)+(p-2)(n+\theta_{3}+\theta_{4})}}\big\},$
or
\begin{align}\label{W013}
&\frac{|Q(mR/2,(R/2)^{p+\vartheta})\cap\{u>\mu^{+}-\frac{\omega}{2^{\kappa_{0}+1+j_{0}}}\}|_{\nu_{w_{1}}}}{|Q(mR/2,(R/2)^{p+\vartheta})|_{\nu_{w_{1}}}}\leq\overline{C}_{0}\sqrt[q]{4}j_{0}^{-\frac{p-q}{pq}}\leq\gamma_{0}.
\end{align}
In light of \eqref{W013} and applying Lemma \ref{LEM0035} with $\sigma=\frac{1}{2^{\kappa_{0}+1+j_{0}}}$, we derive either $\omega\leq \sigma^{-1}AR^{\varepsilon_{0}}=2^{\kappa_{0}+1+j_{0}}AR^{\varepsilon_{0}}$, or
\begin{align*}
u\leq \mu^{+}-\frac{\omega}{2^{\kappa_{0}+2+j_{0}}},\quad\mathrm{in}\text{ $Q(mR/4,(R/4)^{p+\vartheta})$.}
\end{align*}
The proof is finished by taking $\kappa_{\ast}:=\kappa_{0}+1+j_{0}$.

\end{proof}

\subsection{The proofs of Theorems \ref{ZWTHM90} and \ref{THM060}.}
Making use of Corollary \ref{AMZW01}, we establish the following oscillation estimates of the solution.
\begin{prop}\label{PRO30}
Let $R\in(0,\frac{1}{2}]$. Then we have either $\omega\leq 2^{\kappa_{\ast}}AR^{\varepsilon_{0}}$, or
\begin{align}\label{N90}
\mathop{osc}\limits_{Q(mR/4,(R/4)^{p+\vartheta})}u\leq\eta_{\ast}\omega,\quad\eta_{\ast}=1-\frac{1}{2^{\kappa_{\ast}+1}},
\end{align}
where $\kappa_{\ast}$ and $A$ are given by Corollary \ref{AMZW01}.
\end{prop}
\begin{proof}
Observe first that one of the following two inequalities must hold:
\begin{align}\label{WAMQ001}
|B_{mR/2}\cap\{u(\cdot,-(R/2)^{p+\vartheta})>\mu^{+}-2^{-1}\omega\}|_{\mu_{w_{1}}}\leq\frac{1}{2}|B_{mR/2}|_{\mu_{w_{1}}},
\end{align}
and
\begin{align}\label{WAMQ002}
|B_{mR/2}\cap\{u(\cdot,-(R/2)^{p+\vartheta})<\mu^{-}+2^{-1}\omega\}|_{\mu_{w_{1}}}\leq\frac{1}{2}|B_{mR/2}|_{\mu_{w_{1}}}.
\end{align}
It then follows from Corollary \ref{AMZW01} that if $\omega>2^{\kappa_{\ast}}AR^{\varepsilon_{0}}$,
\begin{align*}
\sup\limits_{Q(mR/4,(R/4)^{p+\vartheta})}u\leq\mu^{+}-\frac{\omega}{2^{\kappa_{\ast}+1}},\quad\text{when \eqref{WAMQ001} holds,}
\end{align*}
and
\begin{align*}
\inf\limits_{Q(mR/4,(R/4)^{p+\vartheta})}u\geq\mu^{-}+\frac{\omega}{2^{\kappa_{\ast}+1}},\quad\text{when \eqref{WAMQ002} holds.}
\end{align*}
In either case, we all have
\begin{align*}
\mathop{osc}\limits_{Q(mR/4,(R/4)^{p+\vartheta})}u\leq\Big(1-\frac{1}{2^{k_{\ast}+1}}\Big)\omega.
\end{align*}
The proof is complete.

\end{proof}

Based on Proposition \ref{PRO30}, we are now ready to construct a family of nested and shrinking cylinders with the same vertex, which makes the essential oscillation of $u$ in these cylinders converge to zero as the radius of the cylinder approaches zero. Define
\begin{align}\label{K99}
A_{\ast}:=2^{\kappa_{\ast}}A,\quad \omega_{0}:=\max\{\omega,A_{\ast}R^{\varepsilon_{0}}\},
\end{align}
and, for $k\geq0$,
\begin{align}\label{K98}
R_{k}:=A_{\ast}^{-k}R,\quad \omega_{k+1}:=\max\{\eta_{\ast}\omega_{k},A_{\ast}R_{k}^{\varepsilon_{0}}\},\quad \tilde{a}_{k}:=\Big(\frac{\omega_{k}}{A_{\ast}}\Big)^{\frac{p-2}{p+\vartheta}}.
\end{align}
In light of \eqref{K99}--\eqref{K98}, we have
\begin{align*}
\tilde{a}_{k+1}R_{k+1}=&\Big(\frac{\omega_{k+1}}{A_{\ast}}\Big)^{\frac{p-2}{p+\vartheta}}\frac{R_{k}}{A_{\ast}}\leq\frac{\tilde{a}_{k}R_{k}}{\eta_{\ast}^{\frac{2-p}{p+\vartheta}}A_{\ast}}< \tilde{a}_{k}R_{k},
\end{align*}
which shows that $Q(\tilde{a}_{k+1}R_{k+1},R_{k+1}^{p+\vartheta})\subset Q(\tilde{a}_{k}R_{k},R_{k}^{p+\vartheta})$.
\begin{lemma}\label{LEM83}
For every $k=0,1,2,...,$ we have
\begin{align}\label{OSC98}
\mathop{osc}\limits_{Q(\tilde{a}_{k}R_{k},R_{k}^{p+\vartheta})}u\leq\omega_{k}.
\end{align}

\end{lemma}
\begin{proof}
From \eqref{K98}, we see that $\tilde{a}_{0}\leq a_{0}$. Then \eqref{OSC98} is obviously valid for $k=0$. Assume that \eqref{OSC98} holds for $k=i$ with $i\geq1$. Then we prove that it also holds in the case when $k=i+1$. According to the assumption, we have $\mathop{osc}\limits_{Q(\tilde{a}_{i}R_{i},R_{i}^{p+\vartheta})}u\leq\omega_{i}$. Then applying the proof of Proposition \ref{PRO30} with a slight modification, we obtain
\begin{align}\label{AEM81}
\mathop{osc}\limits_{Q(m_{i}R_{i}/4,(R_{i}/4)^{p+\vartheta})}u\leq\max\{\eta_{\ast}\omega_{i},A_{\ast}R_{i}^{\varepsilon_{0}}\}=\omega_{i+1},\quad m_{i}:=\Big(\frac{\omega_{i}}{\mathcal{M}_{\ast}}\Big)^{\frac{p-2}{p+\vartheta}},
\end{align}
where $\mathcal{M}_{\ast}:=2^{\kappa_{\ast}}\mathcal{M}_{0}.$ Since $\omega_{i+1}\geq\eta_{\ast}\omega_{i}$ and $1<p<2$, then we have
\begin{align*}
&\frac{m_{i}R_{i}}{4}=\Big(\frac{\omega_{i}}{\mathcal{M}_{\ast}}\Big)^{\frac{p-2}{p+\vartheta}}\frac{R_{i}}{4}\geq \Big(\frac{\omega_{i+1}}{A_{\ast}}\Big)^{\frac{p-2}{p+\vartheta}}R_{i+1}\frac{A_{\ast}^{\frac{2p+\vartheta-2}{p+\vartheta}}(\eta_{\ast}\mathcal{M}_{\ast})^{\frac{2-p}{p+\vartheta}}}{4}\geq \tilde{a}_{i+1}R_{i+1}.
\end{align*}
This, together with \eqref{AEM81}, reads that
\begin{align*}
\mathop{osc}\limits_{Q(\tilde{a}_{i+1}R_{i+1},R_{i+1}^{p+\vartheta})}u\leq\omega_{i+1}.
\end{align*}
The proof is finished.

\end{proof}

Making use of Lemma \ref{LEM83} and applying the proof of Proposition 4.11 in \cite{MZ202303} with minor modification, we further obtain the oscillation estimates of $u$ with certain decay rate as follows.
\begin{lemma}\label{PRO90}
Assume as in Theorems \ref{ZWTHM90} and \ref{THM060}. Then for any $0<\rho\leq R\leq\frac{1}{2}$,
\begin{align*}
\mathop{osc}\limits_{Q(\tilde{a}_{0}\rho,\rho^{p+\vartheta})}u\leq \max\big\{\eta_{\ast}^{-1}\omega_{0},A_{\ast}^{1+2\varepsilon_{\ast}}R^{\varepsilon_{\ast}}\big\}\left(\frac{\rho}{R}\right)^{\varepsilon_{\ast}},
\end{align*}
where $\eta_{\ast}$ is defined by \eqref{N90}, $A_{\ast},\omega_{0}$ and $\tilde{a}_{0}$ are, respectively, given by \eqref{K99}--\eqref{K98}, and
\begin{align*}
\varepsilon_{\ast}:=\min\Big\{\varepsilon_{0},-\frac{\ln\eta_{\ast}}{\ln A_{\ast}}\Big\}.
\end{align*}

\end{lemma}

\begin{remark}\label{REM10}
By a translation and applying the proof of Lemma \ref{PRO90} with minor modification, we derive that for any $0<\rho\leq R\leq\frac{1}{2}$ and $t_{0}\in[-\frac{1}{2},0]$,
\begin{align*}
\mathop{osc}\limits_{[(0,t_{0})+Q(\tilde{a}_{0}\rho,\rho^{p+\vartheta})]}u\leq \max\big\{\eta_{\ast}^{-1}\omega_{0},A_{\ast}^{1+2\varepsilon_{\ast}}R^{\varepsilon_{\ast}}\big\}\left(\frac{\rho}{R}\right)^{\varepsilon_{\ast}}.
\end{align*}
\end{remark}

We are now ready to give the proofs of Theorems \ref{ZWTHM90} and \ref{THM060}, respectively.
\begin{proof}[Proof of Theorem \ref{ZWTHM90}]
In light of $1<p<2$, it follows from \eqref{M01} and \eqref{K99} that for any $0<R\leq\frac{1}{2}$,
\begin{align*}
1\leq\left(\frac{A_{\ast}}{\max\{\mathcal{M}_{0},A_{\ast}2^{-\varepsilon_{0}}\}}\right)^{\frac{2-p}{p+\vartheta}}\leq\left(\frac{A_{\ast}}{\max\{\omega,A_{\ast}R^{\varepsilon_{0}}\}}\right)^{\frac{2-p}{p+\vartheta}}=\tilde{a}_{0}.
\end{align*}
Consequently, a combination of Theorem \ref{CORO06}, Lemma \ref{PRO90} and Remark \ref{REM10} leads to that there exists a constant $0<\alpha\leq\varepsilon_{0}$ depending only upon the data such that for any $t_{0}\in(-1/2,0]$ and $\rho\in(0,\frac{1}{2}]$,
\begin{align*}
\mathop{osc}\limits_{[(0,t_{0})+Q(\rho,\rho^{p+\vartheta})]}u\leq C\rho^{\alpha}.
\end{align*}
This yields that for any $(x,t)\in B_{1/2}\times(-1/2,t_{0}]$,

$(i)$ if $|t-t_{0}|\leq2^{-(p+\vartheta)}$, then
\begin{align*}
|u(x,t)-u(0,t_{0})|\leq&|u(x,t)-u(x,t_{0})|+|u(x,t_{0})-u(0,t_{0})|\notag\\
\leq& C(|t-t_{0}|^{\frac{\alpha}{p+\vartheta}}+|x|^{\alpha})\leq C\big(|x|+|t-t_{0}|^{\frac{1}{p+\vartheta}}\big)^{\alpha};
\end{align*}

$(ii)$ if $|t-t_{0}|>2^{-(p+\vartheta)}$, there exists a increasing set $\{t_{i}\}_{i=1}^{\Lambda}$, $1\leq\Lambda\leq[2^{p+\vartheta-1}]+1$ such that $t<t_{1}\leq\cdots\leq t_{\Lambda}<t_{0}$,
\begin{align*}
|u(x,t)-u(0,t_{0})|\leq&|u(x,t)-u(x,t_{1})|+|u(x,t_{1})-u(x,t_{0})|+|u(x,t_{0})-u(0,t_{0})|\notag\\
\leq&C\big(|t-t_{1}|^{\frac{\alpha}{p+\vartheta}}+|t_{1}-t_{0}|^{\frac{\alpha}{p+\vartheta}}+|x|^{\frac{\alpha}{p+\vartheta}}\big)\notag\\
\leq&C\big(|x|+|t-t_{0}|^{\frac{1}{p+\vartheta}}\big)^{\alpha},\quad\text{if}\;\Lambda=1,
\end{align*}
and
\begin{align*}
|u(x,t)-u(0,t_{0})|&\leq|u(x,t)-u(x,t_{1})|+\sum^{\Lambda-1}_{i=1}|u(x,t_{i})-u(x,t_{i+1})|\notag\\
&\quad+|u(x,t_{\Lambda})-u(x,t_{0})|+|u(x,t_{0})-u(0,t_{0})|\notag\\
&\leq C\Big(|t-t_{1}|^{\frac{\alpha}{p+\vartheta}}+\sum^{\Lambda-1}_{i=1}|t_{i}-t_{i+1}|^{\frac{\alpha}{p+\vartheta}}+|t_{\Lambda}-t_{0}|^{\frac{\alpha}{p+\vartheta}}+|x|^{\frac{\alpha}{p+\vartheta}}\Big)\notag\\
&\leq C\big(|x|+|t-t_{0}|^{\frac{1}{p+\vartheta}}\big)^{\alpha},\quad \text{if}\;\Lambda\geq2.
\end{align*}
The proof is complete.

\end{proof}

\begin{proof}[Proof of Theorem \ref{THM060}]
With regard to the detailed proof in the case of $(w_{1},w_{2})=(|x'|^{\theta_{1}},|x'|^{\theta_{3}})$, see \cite{MZ202303}. In the following we give the proof in the case when $(w_{1},w_{2})=(|x|^{\theta_{2}},|x|^{\theta_{4}})$ for readers' convenience. In this case, we have $\theta_{1}=\theta_{3}=0$ and $\vartheta=\theta_{2}-\theta_{4}$. For $R\in(0,1/2)$ and $(y,s)\in Q(1/R,1/R^{p+\vartheta})$, set $u_{R}(y,s)=u(Ry,R^{p+\vartheta}s)$ and $\phi_{3,R}(y,s)=\phi_{3}(Ry,R^{p+\vartheta}s)$. Hence $u_{R}$ satisfies
\begin{align*}
|y|^{\theta_{2}}\partial_{s}u_{R}-\mathrm{div}(|y|^{\theta_{4}}|\nabla u_{R}|^{p-2}\nabla u_{R})=|y|^{\theta_{4}}\left(\lambda_{3}R|\nabla u_{R}|^{p-1}+R^{p}\phi_{3,R}\right),
\end{align*}
for $(y,s)\in Q(1/R,1/R^{p+\vartheta})$.

For any fixed $(x,t),(\tilde{x},\tilde{t})\in B_{1/2}\times(-1/2,0],$ assume without loss of generality that $|\tilde{x}|\leq|x|$. Write $R=|x|$. Making use of the proof of Theorem \ref{ZWTHM90} with a slight modification, we deduce that there exist two constants $0<\gamma<1$ and $C>1$, both depending only upon the data, such that for any given $\bar{y}\in \partial B_{1}$ and $\bar{s}\in(-2^{-1} R^{-(p+\vartheta)},0]$,
\begin{align}\label{WAQA001}
|u_{R}(y,s)-u_{R}(\bar{y},\bar{s})|\leq C\big(|y-\bar{y}|+|s-\bar{s}|^{1/p}\big)^{\gamma},
\end{align}
for any $(y,s)$ satisfying that $|y-\bar{y}|+|s-\bar{s}|^{1/p}\leq1/2$. We now restrict the range of $\gamma$ to be in $(0,\alpha]$. Otherwise, if $\gamma>\alpha$, then \eqref{WAQA001} is also valid for any $\gamma\in(0,\alpha]$.

Observe that
\begin{align}\label{QT01}
|u(x,t)-u(\tilde{x},\tilde{t})|\leq&|u(x,t)-u(x,\tilde{t})|+|u(x,\tilde{t})-u(\tilde{x},\tilde{t})|.
\end{align}
Let $c_{1}\geq1+\frac{p}{p+\vartheta}$. If $|t-\tilde{t}|\leq R^{c_{1}(p+\vartheta)}$, we have from \eqref{WAQA001} that
\begin{align*}
&|u(x,t)-u(x,\tilde{t})|=\left|u_{R}(x/R,t/R^{p+\vartheta})-u_{R}(x/R,\tilde{t}/R^{p+\vartheta})\right|\notag\\
&\leq C|(t-\tilde{t})/R^{p+\vartheta}|^{\frac{\gamma}{p}}\leq C|t-\tilde{t}|^{\frac{(c_{1}-1)\gamma}{c_{1}p}}.
\end{align*}
By comparison, if $|t-\tilde{t}|>R^{c_{1}(p+\vartheta)}$, applying Theorem \ref{ZWTHM90} with $\theta_{1}=\theta_{3}=0$, we derive
\begin{align*}
&|u(x,t)-u(x,\tilde{t})|\notag\\
&\leq|u(x,t)-u(0,t)|+|u(0,t)-u(0,\tilde{t})|+|u(0,\tilde{t})-u(x,\tilde{t})|\notag\\
&\leq C\big(R^{\alpha}+|t-\tilde{t}|^{\frac{\alpha}{p+\vartheta}}\big)\leq C|t-\tilde{t}|^{\frac{\alpha}{c_{1}(p+\vartheta)}}.
\end{align*}

It remains to calculate the second term in \eqref{QT01}. Pick $c_{2}\geq2$. If $|x-\tilde{x}|\leq R^{c_{2}}$, a consequence of \eqref{WAQA001} gives that
\begin{align*}
&|u(x,\tilde{t})-u(\tilde{x},\tilde{t})|=\left|u_{R}(x/R,\tilde{t}/R^{p+\vartheta})-u_{R}(\tilde{x}/R,\tilde{t}/R^{p+\vartheta})\right|\notag\\
&\leq C|(x-\tilde{x})/R|^{\gamma}\leq C|x-\tilde{x}|^{\frac{(c_{2}-1)\gamma}{c_{2}}},
\end{align*}
while, if $|x-\tilde{x}|>R^{c_{2}}$, using Theorem \ref{ZWTHM90} with $\theta_{1}=\theta_{3}=0$ again, we have
\begin{align*}
|u(x,\tilde{t})-u(\tilde{x},\tilde{t})|\leq&|u(x,\tilde{t})-u(0,\tilde{t})|+|u(0,\tilde{t})-u(\tilde{x},\tilde{t})|\notag\\
\leq& C(R^{\alpha}+|\tilde{x}|^{\alpha})\leq C|x-\tilde{x}|^{\frac{\alpha}{c_{2}}}.
\end{align*}
Combining these above facts, we take $c_{1}=1+\frac{p\alpha}{\gamma(p+\vartheta)}$ and $c_{2}=1+\frac{\alpha}{\gamma}$ such that
\begin{align*}
\frac{(c_{1}-1)\gamma}{c_{1}p}=\frac{\alpha}{c_{1}(p+\vartheta)},\quad \frac{(c_{2}-1)\gamma}{c_{2}}=\frac{\alpha}{c_{2}}.
\end{align*}
The choice for the values of $c_{1}$ and $c_{2}$ is actually optimal in the sense that they maximize the H\"{o}lder regularity exponent. This is because $\frac{c_{i}-1}{c_{i}}$ increases in $c_{i}$ and $c_{i}^{-1}$ decreases in $c_{i}$ for $i=1,2.$ Consequently, we deduce that for any $(x,t),(\tilde{x},\tilde{t})\in B_{1/2}\times(-1/2,0),$
\begin{align*}
|u(x,t)-u(\tilde{x},\tilde{t})|\leq C\big(|x-\tilde{x}|+|t-\tilde{t}|^{\frac{1}{p+\vartheta}}\big)^{\frac{\alpha\gamma}{\alpha+\gamma}},
\end{align*}
which indicates that Theorem \ref{THM060} holds.

\end{proof}

%
%%
%\noindent{\bf{\large Conflict of interest.}} The authors declare that they have no conflict of interest.
%\noindent{\bf{\large Data Availability Statement.}} The data used to support the findings of this study are available from the corresponding author upon request.

\noindent{\bf{\large Acknowledgements.}} This work was partially supported by the National Key research and development program of
China (No. 2022YFA1005700). C.X. Miao was partially supported by the National Natural Science Foundation of China (No. 12371095 and 12071043). Z.W. Zhao was partially supported by China Postdoctoral Science Foundation (No. 2021M700358). The authors would like to thank Prof. Liao Naian for pointing out the problem that the constant $\gamma_{0}$ obtained in Lemma \ref{LEM0035} should not depend on the value of the level parameter $\sigma$, which encourages us to solve this problem.

%\noindent{\bf{\large Statements.}}

%\noindent{\bf{\large Acknowledgements.}}

\end{document}